\documentclass[12pt]{article}
 \usepackage[latin1]{inputenc}
 \usepackage[dvips]{graphicx}
 \usepackage{amsmath}
 \usepackage{amsthm}
 \usepackage{amsfonts}
 \usepackage{amssymb}
 \usepackage{layout}
 \usepackage{verbatim}
 \usepackage{alltt}
 \usepackage{xy}
 \input xy
 \xyoption{all}
 
\oddsidemargin - 1.5 pt
\newcounter{theorem}[section]

\renewcommand{\thetheorem}{\arabic{section}.\arabic{theorem}}

\newenvironment{conjecture}{\bgroup{\hspace*{-0.15 cm}\bf{Conjecture}
\thetheorem.}\bgroup\it}{\egroup \egroup\par\medskip}

\newenvironment{theorem}{\par\medskip\noindent\refstepcounter{theorem}
\bgroup{\hspace*{-0.15 cm}\bf{Theorem}
\thetheorem.}\bgroup\it}{\egroup \egroup\par\medskip}

\newenvironment{proposition}{\par\medskip\noindent\refstepcounter{theorem}
\bgroup{\hspace*{-0.15 cm}\bf{Proposition}
\thetheorem.}\bgroup\it}{\egroup \egroup\par\medskip}

\newenvironment{definition}{\par\medskip\noindent\refstepcounter{theorem}
\bgroup{\hspace*{-0.15 cm}\bf{Definition}
\thetheorem.}\bgroup}{\egroup \egroup\par\medskip}

\newenvironment{question}{\par\medskip\noindent\refstepcounter{theorem}
\bgroup{\hspace*{-0.15 cm}\bf{Question} \thetheorem.}\bgroup}{\egroup
\egroup\par\medskip}

\newenvironment{remark}{\par\medskip\noindent\refstepcounter{theorem}
\bgroup{\hspace*{-0.15 cm}\bf{Remark} \thetheorem.}\bgroup}{\egroup
\egroup\par\medskip}

\newcommand{\C}{\mathbb C}

\newcommand{\R}{\mathbb R}
\newcommand{\Q}{\mathbb Q}
\newcommand{\Z}{\mathbb Z}
\newcommand{\F}{\mathbb F}

\providecommand{\abs}[1]{\lvert#1\rvert}

\def\p{{\mathfrak p}}
\def\P{{\mathbb P}}
\def\H{{\mathbb H}}
\def\pgl{{\rm PGL}}
\def\psl{{\rm PSL}}

\title{On the Integral Cohomology of Bianchi groups}  
\author{Mehmet Haluk \c{S}eng\"un}
\date{}

\begin{document}
\maketitle
\abstract{Extensive and systematic machine computations are carried out to 
investigate the integral cohomology of the Euclidean Bianchi groups and their congruence subgroups. 
The collected data give insight on several aspects, including the asymptotic behaviour of the 
torsion in the first homology. Along with the experimental work, some basic properties of the 
integral cohomology are recorded with an eye towards the liftability issue of Hecke eigenvalue systems.}


\section{Introduction}
Bianchi groups are groups of the form $\psl_2(\mathcal{O})$ where $\mathcal{O}$ is the ring of integers of an imaginary quadratic field. 
First studied by L.Bianchi in 1892 \cite{bianchi}, these groups form an important class of arithmetic Kleinian groups. In fact, it is well known that any non-cocompact arithmetic Kleinian group, after conjugation, is commensurable with a Bianchi group.

In this paper, I will be interested in the cohomology of Bianchi groups with certain $\mathcal{O}$-module coefficients. These cohomology groups are 
fundamental to the study of Bianchi modular forms, that is, modular forms (for {\bf GL}$_2$) over an imaginary quadratic field. Unlike 
their analogs over totally real fields, i.e. Hilbert modular forms, the arithmetic of Bianchi modular forms is little understood. One of the 
features that obstruct the application of standard methods is the torsion in the cohomology of Bianchi groups. 

The first computations of torsion were performed in 1981 by Elstrodt, Grunewald and Mennicke \cite{egm} where they computed the abelianizations $\Gamma_0(\mathfrak{p})^{ab}$ ($\simeq H_1(\Gamma_0(\mathfrak{p}),\Z)$) of congruence subgroups $\Gamma_0(\mathfrak{p})$ for prime ideals $\mathfrak{p}$ of $\Z[i]$ of residue degree 1 and norm $\leq 400$. In the same paper, numerical evidence suggesting a connection between some of the 2-torsion classes and certain $S_3$-extension of $\Q(i)$ was exhibited. Later in 1994, Figueiredo \cite{figu} provided several more examples which suggested the same kind of connection between certain 3-torsion classes and certain $\textrm{SL}_2(\F_3)$-extensions of the fields $Q(\sqrt{-d})$ with $d=1,2,3$. Recently further examples have been found by myself in \cite{sen_a5} and by Torrey in \cite{torrey} where she formulates and tests an analogue of the, now proven, strong modularity conjecture of Serre \cite{serre}. 

In 1983, Grunewald and Schwermer \cite{gs} determined the conjugacy classes of small index subgroups of several Bianchi groups. In particular, they computed the abelianizations of a vast amount of small index ($\leq$ 12) subgroups. Depending on the data they collected and the data in \cite{egm}, they speculated that for any finite index subgroup $\Gamma$ of a Bianchi group $G$, if a prime $p$ appears as an  exponent for an element of $\Gamma^{ab}$, then $p$ should be smaller than half of the index of $\Gamma$ in $G$.   

In 1988, Taylor in his unpublished thesis \cite{taylor} proved the existence of  $p^m$-torsion classes in $H_1(\Gamma_1(Np^r), E_{k,\ell}(\mathcal{O}))$ (notation will be explained later) with $k \not = \ell$ under special circumstances. In her unpublished 2000 thesis \cite{prip}, Priplata numerically investigated the torsion for some Bianchi groups, mostly for coefficient modules $E_{k,\ell}$ where $k \not = \ell$. It is worth remarking that the (co)homology of Bianchi groups is related to cuspidal Bianchi modular forms only in the parallel weight case, that is, in the case $k = \ell$. 

In this paper I report on my extensive systematic computations of the integral (co)homology of Bianchi groups. More specifically, I worked with the groups $\psl(\mathcal{O}_d)$ and $\pgl_2(\mathcal{O}_d)$ with $-d=1,2,3,7,11$ and also with their congruence subgroups. Motivated by the possibility of congruences between cuspidal Bianchi modular forms and torsion classes, and also the arithmetical connections mentioned above, I limited myself to the (co)homology with parallel weights. The data I collected 
\begin{itemize}
\item show that $H^2_{cusp}$ can have sporadically large torsion part with very little torsion-free part, in particular the speculations of Grunewald and Schwermer mentioned above are false, 
\item support a recent conjecture of Long, Maclachlan and Reid \cite{lmr} on existence of certain families of rational homology spheres, 
\item strongly suggest that an analogue of the very recent result of Bergeron-Venkatesh \cite{bv} on the asymptotic behaviour of the torsion in the homology of cocompact arithmetic congruence lattices in $\textrm{SL}_2(\C)$ holds for Bianchi groups. 
\end{itemize}
Along with the computational work, I record some basic properties of the integral cohomology with an eye towards the liftability issue of Hecke eigenvalue systems. \\

{\bf Acknowledgments} I started this work towards the end of my postdoctoral fellowship within the Arithmetic Geometry Research Group of the University of Duisburg-Essen under SFB/TR 45 of the Deutsche Forschungsgemeinschaft. I am grateful for the wonderful ambiance and hospitality provided by the Mathematics Department. I gratefully acknowledge that all the computations were carried out on the computers of the Institute for Experimental Mathematics, Essen. I thank the excellent Hausdorff Institute for Mathematics in Bonn where I completed the rest of this work as a visitor. I would like to thank Gebhard B\"ockle and Gabor Wiese for enlightening discussions and many comments/corrections on this work. My correspondence with Nicolas Bergeron and Akshay Venkatesh was very encouraging and gave me more directions for experimentation. I am grateful to Nansen Petrosyan for his help with the spectral sequences. Finally I thank J-P. Serre and the referees whose numerous comments improved the paper. \\

{\bf Dedication} After the first draft of this work was completed, {\it Fritz Grunewald}, whom I consider my second PhD adviser, tragically passed away. Over the years, I benefited greatly from his invaluable guidance and generous support. I dedicate this paper to his memory with admiration, gratitude and love.
 
\section{The modules}

Given a commutative a ring $R$, let $E_k(R) \simeq R[x,y]_{k}$ where the latter is the space of homogeneous degree $k$ polynomials in two variables over $R$. Note that $\lbrace x^{k-i}y^i : 0 \leq i \leq k \rbrace $ is an $R$-basis of $E_k(R)$. 

For a polynomial $P(X,Y)$ in $E_k(R)$ and a matrix $\bigl( \begin{smallmatrix} a & b  \\ c & d \\ \end{smallmatrix} \bigr)$ in 
$\textrm{M}_2(R)$, the we have the right action
 $$  \bigl ( P \cdot  \bigl( \begin{smallmatrix} a & b  \\ c & d \\ \end{smallmatrix} \bigr)\bigr ) \bigl( X,Y \bigr )
 = P \bigl ( \bigl( \begin{smallmatrix} a & b  \\ c & d \\ \end{smallmatrix} \bigr) \bigl( \begin{smallmatrix} X  \\ Y \\ \end{smallmatrix} \bigr)  \bigr )= P \bigl( a X+b Y,c X+d Y \bigr ).$$
 
Let $\mathcal{O}$ be the ring of integers of an imaginary quadratic field. Consider the $\textrm{M}_2(\mathcal{O})$-module 
$$E_{k,\ell}(\mathcal{O}):=E_k (\mathcal{O})\otimes_{\mathcal{O}} \overline{E_{\ell}(\mathcal{O})}.$$
Here the overline on the second factor is to indicate that action on the second factor is twisted with complex conjugation. Note that we should insist that $k+\ell$ is even so that $-Id$ acts trivially and thus $\psl(2,\mathcal{O})$ acts on it as well. \\

It is useful to remark that $E_{k,\ell}(\mathcal{O}) \simeq \text{Sym}^k(\mathcal{O}^2) \otimes_{\mathcal{O}} \overline{\text{Sym}}^{\ell}(\mathcal{O}^2)$ as $\textrm{M}_2(\mathcal{O})$-modules where $\text{Sym}^i(\mathcal{O}^2)$ is the $i$th symmetric power of the standard representation of $\textrm{M}_2(\mathcal{O})$ on $\mathcal{O}^2$. Here the overline on the 
second factor means that the action is twisted with complex conjugation.

Let $\pi$ be a prime element of $\mathcal{O}$ over a rational prime $p$. Put $\kappa_{\pi}$ for its residue field. We put 
$$E_{k,l}(\kappa_{\pi}) :=E_{k,\ell}(\mathcal{O}) \otimes_{\mathcal{O}} \kappa_{\pi}.$$
If $p$ splits in $\mathcal{O}$, then 
$$ E_{k,l}(\kappa_{\pi}) \simeq E_k (\kappa_{\pi})\otimes E_{\ell}(\kappa_{\bar{\pi}}).$$
Thus $\psl_2(\mathcal{O})$ acts on it by reduction mod $\pi$ on the first factor and by reduction mod $\bar{\pi}$ on the second.
If $p$ is inert in $\mathcal{O}$, then 
$$ E_{k,l}(\kappa_{\pi}) \simeq E_k (\kappa_{\pi})\otimes E_{\ell}(\kappa_{\pi})^{\sigma}.$$
Here $\psl_2(\mathcal{O})$ acts by reduction mod $\pi$. The action on the second factor is twisted by the nontrivial automorphism $\sigma$ of $\kappa_{\pi}$.

Finally, when $p$ is ramified in $\mathcal{O}$, we have
$$ E_{k,l}(\kappa_{\pi}) \simeq E_k (\kappa_{\pi})\otimes E_{\ell}(\kappa_{\pi}).$$
Here the action of $\psl_2(\mathcal{O})$ is via reduction mod $\pi$ and is the same on both factors.

A result of Brauer and Nesbitt \cite{bn} tells us that in the inert and split cases, the $\psl_2(\mathcal{O})$-modules $E_{k,\ell}(\kappa_{\pi})$ are irreducible only when $0 \leq k,\ell \leq p-1.$ In the ramified case, $E_{k,l}(\kappa_{\pi})$ is never an irreducible $\psl_2(\mathcal{O})$-module unless $k=0 \leq \ell \leq p-1$ or  $\ell =0 \leq k \leq p-1$. For more on the structure of these modules, we refer to reader to \cite{st}.

The following will be used later.
\begin{proposition}\label{pairing} Let $\mathcal{O}$ be the ring of integers of an imaginary quadratic field. Let $k \geq \ell$ and put $R=\mathcal{O}[\frac{1}{k!}]$. Then there is a $\psl_2(\mathcal{O})$-equivariant perfect pairing 
$$E_{k,\ell}(R) \times E_{k,\ell}(R) \rightarrow R.$$
\end{proposition}
It is well known (see e.g. \cite{wi} Lemma 2.4) that there is a perfect pairing on $\text{Sym}^n(R^2)$ coming from the determinant pairing on $R^2$ whenever $n!$ is invertible in the ring $R$. The proposition follows by taking the product of the two pairings associated to the two factors of $E_{k,\ell}$. For an explicit description of this pairing, see Section 2.4. of \cite{berger}. As a corollary, we see that the modules $E_{k,\ell}(R)$ are self-dual.

\section{the cohomology}
In this section I will investigate the integral cohomology of Bianchi groups. My treatment is heavily influenced by work of Hida \cite{hida}, Wang \cite{wang} and Wiese \cite{wi}.
  
Let $K$ be an imaginary quadratic field. Let $\mathcal{O}$ be its ring of integers. Let $G$ be the associated Bianchi group. 
Let $\Gamma$ be a finite index subgroup of $G$.

In this paper, we will focus on the $\mathcal{O}$-modules 
$$H^i(\Gamma,E_{k,l}(\mathcal{O})), \ \ \ \ i=1,2.$$
It is well known that these are finitely generated $\mathcal{O}$-modules.

\begin{definition} Let $\pi \in \mathcal{O}$ be a prime element over the rational prime $p$. Assume that $H^i(\Gamma, E_{k,\ell}(\mathcal{O}))$ has $\pi$-torsion, i.e. it contains a non-zero class $c$ such that $\pi \cdot c = 0$. We say that $\pi$ is a {\em large torsion} if $k, \ell < p$. Otherwise, we say that $\pi$ is a {\em small torsion}.
\end{definition}

\begin{proposition} \label{main} Let $\pi$ be prime element of $\mathcal{O}$ over the rational prime $p$. Put $\kappa_{\pi}$ for its residue field. 
Let $\Gamma$ be a torsion-free finite index subgroup of the Bianchi group $G$. 
\begin{enumerate}
\item[(a)]  If $\Gamma$ surjects onto $\psl_2(\kappa_{\pi})$ and $\pi$ is unramified, then  $H^1(\Gamma, E_{k,\ell}(\mathcal{O}))$ has no large $\pi$-torsion.  
\item[(b)]  If $\Gamma$ surjects onto $\psl_2(\kappa_{\pi})$ and $\pi$ is ramified, then $H^1(\Gamma, E_{k,\ell}(\mathcal{O}))$ has no $\pi$-torsion if and only if $k=0$ and $\leq \ell \leq p-1$ or vice versa.
\item[(c)]  The obstruction to the lifting of a class in $H^1(\Gamma, E_{k,\ell}(\kappa_{\pi}))$ to $H^1(\Gamma, E_{k,\ell}(\mathcal{O}))$ \\ is the $\pi$-torsion in $H^2(\Gamma,E_{k,\ell}(\mathcal{O}))$.
\item[(d)] $H^2(\Gamma, E_{k,\ell}(\mathcal{O})) \otimes \kappa_{\pi} \simeq H^2(\Gamma, E_{k,\ell}(\kappa_{\pi}))$ for every $k,\ell$.
\end{enumerate}
\end{proposition}

\begin{proof}   
In the following, let us put $E=E_{k,\ell}$.
Consider the following short exact sequence 
$$ 0 \rightarrow E(\mathcal{O}) \xrightarrow{\cdot \pi} E(\mathcal{O}) \rightarrow E(\kappa_{\pi}) \rightarrow 0 .$$
where $\cdot \pi$ is the multiplication by $\pi$ map.

The associated long exact sequence gives the following short exact sequence
$$ 0 \rightarrow H^i(\Gamma, E(\mathcal{O})) \otimes \kappa_{\pi} \rightarrow H^i(\Gamma, E(\kappa_{\pi})) \rightarrow H^{i+1}(\Gamma, E(\mathcal{O}))[\pi] \rightarrow 0$$
for $i \geq 0$. Here $H^j(\Gamma, E(\mathcal{O}))[\pi]$ denotes the kernel of the map induced by $\cdot \pi$. \\
Putting $i=0$, we get
$$E(\kappa_{\pi})^{\Gamma} \simeq H^1(\Gamma, E(\mathcal{O}))[\pi]$$
Now (a) and (b) follow via the irreducibility discussions of the previous section.
For $i=1$ we get
$$0 \rightarrow H^1(\Gamma, E(\mathcal{O})) \otimes \kappa_{\pi} \rightarrow H^1(\Gamma, E(\kappa_{\pi})) \rightarrow H^2(\Gamma, E(\mathcal{O}))[\pi] \rightarrow 0$$
which explains the claim (c).
It is known that the virtual cohomological dimension of a Bianchi group is 2. Setting $i=2$, we get
$$H^2(\Gamma, E(\mathcal{O})) \otimes \kappa_{\pi} \simeq H^2(\Gamma, E(\kappa_{\pi})),$$
finishing the proof.
\end{proof}

Each cohomology space comes equipped with a commuting family $\mathbb{T}$ of Hecke operators acting on it, see \cite{st}. A \textit{(Hecke) eigenvalue system} with values in a ring $R$ is a ring homomorphism $\Phi : \mathbb{T} \rightarrow R$. We say that an eigenvalue system $\Phi$ occurs in an $R\mathbb{T}$-module $A$ if there is a nonzero element $a \in A$ such that $Ta=\Phi(T)a$ for all $T$ in $\mathbb{T}$. Using a lifting theorem of Ash and Stevens (\cite{as}, Prop.1.2.2), we can lift an eigenvalue systems occurring in $H^2(\Gamma, E(\kappa_{\pi}))$ to one occurring in $H^2(\Gamma, E(\mathcal{R}))$ where $\mathcal{R}$ is some finite extension of the completion of $\mathcal{O}$ at $\pi$. The possible $p$-torsion in $H^2(\Gamma, E(\mathcal{O}))$ obstructs us from applying the lifting theorem to lift eigenvalue systems occurring in $H^1(\Gamma, E(\kappa_{\pi}))$.

\subsection{Cuspidal cohomology}
There is a subspace of the cohomology that is of special interest due to the fact that it can be 
identified with cuspidal Bianchi modular forms. 

Let $K=\Q(\sqrt{-d})$ be an imaginary quadratic field of class number $h_K$ with ring of integers $\mathcal{O}=\mathcal{O}_d$. Let $\mathbb{P}$ denote the projective line over $K$ and 
$G=\psl_2(\mathcal{O})$. The group $\psl_2(K)$ acts naturally on $K^2$ and thus on $\P$. It is well-known that the cardinality $\abs{\P / G}$ of the set $\P / G$ of $G$-orbits of $\P$ is equal to $h_K$. Hence $\abs{\P / \Gamma}$ is finite for any finite index subgroup $\Gamma$ of $G$. We will call the elements of $\P / \Gamma$ the {\it cusps} of $\Gamma$.

For every $D \in \P$, let $B_D$ be the Borel subgroup of $G$ defined by the (setwise) stabilizer of $D$ in $G$. Then the pointwise stabilizer of $D$ in $G$ is the unipotent radical $U_D$ of the Borel subgroup $B_D$. Let $\Gamma$ be a finite index subgroup of $G$ and $D_c$ be a representative for a cusp $c$ of $\Gamma$. Define 
$$\Gamma_c := B_{D_c} \cap \Gamma.$$
If $\Gamma_c$ is torsion-free (this is automatic if $\Gamma$ is itself torsion-free or $-d \not = 1,3$), then  $\Gamma_c = U_{D_c} \cap \Gamma$ and $\Gamma_c$ is free abelian of rank two (see \cite{serre70} p.507). The group 
$$ U(\Gamma) := \bigoplus_{c \in \P / \Gamma} \Gamma_c$$
is independent of the choice of representatives taken for the cusps of $\Gamma.$

Let $E$ be a $\Gamma$-module. Consider the long exact sequence of relative group cohomology for the pair $(\Gamma, U(\Gamma))$
$$ \hdots \rightarrow H^{i-1}_c(\Gamma,E) \rightarrow H^i(\Gamma,E) \rightarrow H^i(U(\Gamma),E) \rightarrow \hdots $$
where $H^n_c(\Gamma,E):=H^n(\Gamma;U(\Gamma),E)$ and the third arrow is given by the restriction maps.

\begin{definition} The {\it cuspidal cohomology} $H^i_{cusp}(\Gamma,E)$
is defined as the image of the cohomology with compact support in $H^i(\Gamma,E)$, or equivalently as the kernel of the restriction map 
$H^i(\Gamma,E) \rightarrow H^i(U(\Gamma),E)$.   
\end{definition}

\begin{remark}\label{esh} Let $S_{k,\ell}(\Gamma)$ denote the space of cuspidal Bianchi modular forms with level $\Gamma$ and weight $(k,\ell)$. Harder proved in \cite{harder} the so called Eichler-Shimura-Harder isomorphism 
$$S_{k,\ell}(\Gamma) \simeq H^1_{cusp}(\Gamma,E_{k,\ell}(\C)) \simeq H^2_{cusp}(\Gamma,E_{k,\ell}(\C))$$
of Hecke modules. Note that the second isomorphism is an instance of a duality result which says that 
if $F$ is a field in which 6 is invertible, then  
$$H^1_{cusp}(\Gamma, E(F))^{\vee} \simeq H^2_{cusp}(\Gamma, E(F)^{\vee })$$
as Hecke modules where $-^{\vee }$ denotes the dual, see \cite{as} Lemma 1.4.3.

Deep results of Borel and Wallach (see Section II of their book \cite{bw}) imply that whenever $k \not= \ell$, the cuspidal 
cohomology $H^i_{cusp}(\Gamma,E_{k,\ell}(\C))$ vanishes. It is important to remark that this is not true anymore when the module 
$E_{k,\ell}$ is not over a field of characteristic 0. In particular, $H^i_{cusp}(\Gamma,E_{k,\ell}(\mathcal{O}))$ is completely torsion 
when $k \not= \ell$. 
\end{remark}

\begin{proposition} \label{infinitytorsion}
Let $\Gamma$ be a finite index subgroup of the Bianchi group $\psl_2(\mathcal{O})$. Assume either that $\Gamma$ is torsion-free or that $-d \not = 1,3$. Then
$H^2(U(\Gamma), E_{k,\ell}(\mathcal{O}))$ has no large torsion.
\end{proposition}
\begin{proof} It is enough to prove the claim for a single cusp $c$ of $\Gamma$, that is, for $H^2(\Gamma_c,E)$ . So fix a cusp $c$ and $\Gamma_c$.  
Let $E=E_{k,\ell}$ and $t=\text{max}\{k,\ell \}$. Put $R=\mathcal{O}[\frac{1}{t!}]$. Composition of the cup product and the perfect pairing of Proposition \ref{pairing} gives us a pairing 

$$\xymatrix{ H^0(\Gamma_c, E(R)) \times H^2(\Gamma_c, E(R)) \ar[r]^{\ \ \ \ \ \cup} & 
             H^2(\Gamma_c,E(R) \otimes_{R} E(R)) \ar[d]^{(\cdot, \cdot)}  \\ 
            &  H^2(\Gamma_c, R ) \simeq R.}$$
That $H^2(\Gamma_c, R ) \simeq R$ can be shown as follows. Recall that $\Gamma_c$ is free abelian with two generators, say $a,u$. 
It is known, see \cite{maclane} p.188, that the tensor product of the two resolutions 
$$\xymatrix{ 0 \ar[r] & R[\langle a \rangle ] \ar[r]^{1-a}& R[\langle a \rangle] \ar[r]^{\varepsilon} & R \ar[r] & 0, \\
0 \ar[r] & R[\langle u \rangle ] \ar[r]^{1-u}& R[\langle u \rangle] \ar[r]^{\varepsilon} & R \ar[r] & 0, 
}$$
where $\varepsilon$ is the usual augmentation map, gives a resolution of $\Gamma_c$. One sees from this resolution that the second cohomology of $\Gamma_c$ with any (right) $R$-module $M$ can be described as 
$$H^2(\Gamma_c,M) \simeq M / \left ( M(1-a)+ M(1-u) \right ) .$$
In the case of trivial module $R$, it follows immediately that $H^2(\Gamma_c, R ) \simeq R$.
          
The above pairing gives that $H^2(\Gamma_c, E(R)) \simeq H^0(\Gamma_c,E(R))^{\vee }$. Clearly $H^0(\Gamma_c, E(R)) \simeq E(R)^{\Gamma_c}$ 
is torsion-free. This implies that its dual and hence $H^2(\Gamma_c, E(R))$ is torsion-free. The claim that there can only be small torsion in $H^2(\Gamma_c, E(\mathcal{O}))$ now follows as $R=\mathcal{O}[\frac{1}{t!}].$ 
\end{proof}

As a corollary we see that the cuspidal part of $H^2$ is responsible for the possible large torsion. 
The referee brought to my attention that the analogue of the above result for $H^1(U(\Gamma),E_{k,k}(\mathcal{O}))$ was proven by Urban in \cite{urban} Prop.$2.4.1$. 
This resut of Urban similarly implies that the possible large torsion in $H^2_c(\Gamma,E_{k,k}(\mathcal{O}))$ comes from  $H^2_{cusp}(\Gamma,E_{k,k}(\mathcal{O}))$ as well.

\begin{proposition} \label{h2cuspidal} Let $G$ be the Bianchi group $\psl_2(\mathcal{O})$. 
Let $\pi$ be prime element of $\mathcal{O}$ over the rational prime $p$ and put $\kappa_{\pi}$ for its residue field. 
Let $\Gamma$ be a torsion-free finite index subgroup of $G$. Then 
$$H^2_{cusp}(\Gamma, E_{k,\ell}(\mathcal{O})) \otimes \kappa_{\pi} \simeq H^2_{cusp}(\Gamma, E_{k,\ell}(\kappa_{\pi}))$$
for every $k,\ell < p$.
\end{proposition}
\begin{proof} Put $E=E_{k,\ell}(\mathcal{O})$. Now consider the commutative diagram 

$$\xymatrix{ & H^2(\Gamma,E) \ar[d] \ar[r]^{\cdot \pi} & H^2(\Gamma,E) \ar[d] \ar[r] & H^2(\Gamma,E(\kappa_{\pi})) \ar[d] \ar[r] & 0 \\ 
0 \ar[r] & H^2(U(\Gamma),E) \ar[r]^{\cdot \pi} & 
           H^2(U(\Gamma),E) \ar[r] & 
           H^2(U(\Gamma),E(\kappa_{\pi})) }$$
Here the vertical maps are given by the usual restriction maps.

The horizontal lines are exact. The exactness of the first line comes from Proposition \ref{main} part (d). The exactness 
of the second line amounts to Proposition \ref{infinitytorsion}.

Observe that the cokernel of the restriction map $H^2(\Gamma,E) \rightarrow H^2(U(\Gamma),E)$ is isomorphic to $H^3_{cusp}(\Gamma,E)  
\subset H^3(\Gamma,E)$. Since the virtual cohomological dimension of a Bianchi group is two and $\Gamma$ is torsion-free, we have 
$H^3(\Gamma,E)=0$.  Now the claim follows by the Snake Lemma. 
\end{proof}

Let us end this section with the following observation on lifting eigenvalue systems.
 \begin{proposition} \label{liftsolution} Let $\pi \in \mathcal{O}$ be a prime element over the rational prime $p>3$.
Let $\Gamma$ be a torsion-free finite index subgroup of $\psl_2(\mathcal{O})$. Let $\Phi$ be an eigenvalue system 
occuring in $H^1_{cusp}(\Gamma, E_{k,\ell}(\kappa_{\pi}))$ with $k,\ell <p$ and $\kappa_{\pi}=\mathcal{O}/(\pi)$. If $\Phi$ does not lift to 
$H^1_{cusp}(\Gamma, E_{k,\ell}(\mathcal{R}))$ for any finite extension $\mathcal{R}$ of the completion 
$\mathcal{O}_{\pi}$ of $\mathcal{O}$ at $\pi$, then there is a $\pi$-torsion eigenclass $c \in H^2_{cusp}(\Gamma,E_{k,\ell}(\mathcal{O}))$ realizing a lift of $\Phi$.
 \end{proposition} 
\begin{proof} 
As $p>3$, by the duality result I mentioned in Remark \ref{esh}, we deduce that $\Phi^{\vee}$ lives in $H^2_{cusp}(\Gamma, E_{k,\ell}(\kappa_{\pi}))$. Note that our coefficient modules are self-dual. Since $k,\ell < p$, using Proposition \ref{h2cuspidal} and the lifting theorem of Ash and Stevens mentioned after Prop.\ref{main}, we infer that there is an eigenvalue system $\Psi$ living in $H^2_{cusp}(\Gamma,E_{k,\ell}(\mathcal{R}))$ lifting $\Phi^{\vee}$ for some finite extension $\mathcal{R}$ of $\mathcal{O}_{\pi}$. 
If $\Psi$ is not realized by a torsion eigenclass $c \in H^2_{cusp}(\Gamma,E_{k,\ell}(\mathcal{R}))$, 
then we can realize $\Psi$ in  $H^2_{cusp}(\Gamma,E_{k,\ell}(L))$ where $L$ is the field of fractions of $\mathcal{R}$. By duality again, $\Psi^{\vee}$ occurs in $H^1_{cusp}(\Gamma,E_{k,\ell}(L))$. As $\Psi^{\vee}$ has integral values, it can be realized in 
$H^1_{cusp}(\Gamma,E_{k,\ell}(\mathcal{R}))$. Clearly $\Psi^{\vee}$ is a lift of $\Phi$ and this contradicts our starting assumption of non-liftability. To finish, observe that $H^i_{cusp}(\Gamma,E_{k,\ell}(\mathcal{R})) \simeq H^i_{cusp}(\Gamma,E_{k,\ell}(\mathcal{O})) \otimes_{\mathcal{O}} \mathcal{R}$.  
\end{proof}

\section{First cohomology}

I will now describe a method, first observed by Fox \cite{f}, that allows us to compute $H^1$ of any finitely presented group with coefficients in a finite dimensional module. It is well known that Bianchi groups are finitely presented. Presentations for many Bianchi groups are in the literature, see, for example, \cite{fgt}.

Let me illustrate the method through an example. A formal exposition is contained in \cite{fgt}. Let $w=\sqrt{-2}$ and $G=\psl_2(\Z[w])$. It is known that  
$$G= <A,B,U \mid B^2=(AB)^3=[A,U]=(BU^2BU^{-1})^2=1>$$
where $A,B,U$ can be realized as $\bigl( \begin{smallmatrix} 1 & 1  \\ 0 & 1 \\ \end{smallmatrix} \bigr),
 \bigl( \begin{smallmatrix} 0 & -1  \\ 1 & 0 \\ \end{smallmatrix} \bigr), 
\bigl( \begin{smallmatrix} 1 & w  \\ 0 & 1 \\ \end{smallmatrix} \bigr)$ respectively. 

Let $E$ be any $G$-module.
Given any cocycle $f: G \rightarrow E$, any value $f(X)$ can be expressed linearly in terms of the images $f(A),f(B),f(U)$
of the generators of $G$, e.g. $$f(ABU)=f(A)\cdot BU+f(B) \cdot U+f(U).$$ 
Moreover, $f(A),f(B),f(U)$ satisfy the linear equations coming from the relations of the presentation. For example, 
$$B^2=1 \Longrightarrow f(B)(B+1)=0.$$
Conversely, any pair $(x,y,z) \in E^3$ satisfying the linear equations coming the presentation gives uniquely a cocycle. 
Thus the space of cocycles can be seen as the kernel of the matrix corresponding to this linear system. 
One gets the coboundaries similarly and hence computes $H^1(G,E)$ as the quotient of the two spaces. 

Note that to compute with a finite index subgroup $\Gamma$ of $G$, it is not practical to apply the method to a presentation 
of $\Gamma$ (which can be derived from that of $G$ once the coset representatives are known). 
It is best to use Shapiro's lemma and compute $H^1(G,\textrm{Coind}_{\Gamma}^G(E))$.

\subsection{Data on the integral first cohomology}

I have implemented the above algorithm in MAGMA \cite{magma} for the five Euclidean imaginary quadratic fields $K=\Q(\sqrt{-d})$ with 
$d=1,2,3,7,11.$ In the following, let $\mathcal{O}_d$ denote the corresponding ring of integers.

By the theory of modules over principal ideal domains, we know that our $\mathcal{O}$-module $H^1(\Gamma, E(\mathcal{O}))$ has a 
decomposition
$$H^1(\Gamma, E(\mathcal{O})) \simeq \mathcal{O} / (a_1) \oplus \hdots \oplus \mathcal{O} / (a_m) \oplus \mathcal{O}^r.$$
with $a_i \not = 0,1$ and $a_i | a_{i+1}$. The $a_i$ are called elementary divisors and are unique up to multiplication by units. The exponent $r$ is called the rank.

Below I report on some of my computations. 
Observe that the torsion is always ``small" as proved in Proposition \ref{main}. The only exception to this 
is the ramifying prime which always appears in the torsion. We show the rank in a 
separate column as it provides a means to check our work against the dimension computations of \cite{fgt}.  

\begin{center}
\begin{tabular}{|r|c|l|c|} 
\multicolumn{4}{c}{data for $H^1(\psl_2(\mathcal{O}_1),E_{n,n}(\mathcal{O}_{1}))$ } \\ \hline
$n$ & norms of elt. divisors & primes & rank \\ \hline
0&                                                     &&0 \\ \hline  
1& [ 4]                                                &(2)&1 \\ \hline
2& [ 2, 16 ]                                           &(2)&0 \\ \hline
3& [ 2, 2, 4]                                          &(2)&1 \\ \hline
4& [ 2, 2, 2, 8, 1152 ]                                &(2,3)&0   \\ \hline
5& [ 2, 2, 2, 2, 4, 4 ]                                &(2)&2 \\ \hline
6& [ 2, 2, 2, 2, 2, 2, 8, 8, 800 ]                     &(2,5)&0 \\ \hline
7& [ 2, 2, 2, 2, 2, 2, 2, 2, 4, 4, 4]                  &(2)&3 \\ \hline
8& [ 2, 2, 2, 2, 2, 2, 2, 2, 2, 4, 8, 8, 32, 225792 ]  &(2,3,7)&0 \\ \hline
9& [ 2, 2, 2, 2, 2, 2, 2, 2, 2, 2, 2, 2, 4, 4, 4, 4, 4 ] & (2)& 3 \\ \hline 
10& [ 2, 2, 2, 2, 2, 2, 2, 2, 2, 2, 2, 2, 2, 2, 2, 8, 8, 8, 8, 16, 288] &(2,3)& 1 \\ \hline
\end{tabular}
\end{center}

\vspace{.1 in}

\begin{center}
\begin{tabular}{|r|c|l|c|} 
\multicolumn{4}{c}{data for $H^1(\psl_2(\mathcal{O}_2),E_{n,n}(\mathcal{O}_{2}))$ } \\ \hline
$n$ & norms of elt. divisors & primes & rank \\ \hline
1& [ 8]                                                     &(2)&1  \\ \hline
2& [ 2, 32]                                                 &(2)&1   \\ \hline
3& [ 2, 2, 8]                                               &(2)& 2   \\ \hline
4& [ 2, 2, 2, 8, 1152]                                      &(2,3)&1   \\ \hline
5& [ 2, 2, 2, 2, 8, 8]                                      &(2)&3   \\ \hline
6& [ 2, 2, 2, 2, 2, 2, 8, 8, 7200]                          &(2,3,5)&2  \\ \hline
7& [ 2, 2, 2, 2, 2, 2, 2, 2, 8, 8, 8]                       &(2)&4   \\ \hline
8& [ 2, 2, 2, 2, 2, 2, 2, 2, 2, 4, 8, 8, 32, 225792]        &(2,3,7)&2   \\ \hline
9& [ 2, 2, 2, 2, 2, 2, 2, 2, 2, 2, 2, 2, 8, 8, 8, 8, 8]     &(2)&5  \\ \hline
10& [ 2, 2, 2, 2, 2, 2, 2, 2, 2, 2, 2, 2, 2, 2, 2, 8, 8, 8, 8, 32, 288] &(2,3)&3  \\ \hline
\end{tabular}
\end{center}

\vspace{.1 in}

\begin{center}
\begin{tabular}{|r|c|l|c|} 
\multicolumn{4}{c}{data for $H^1(\psl_2(\mathcal{O}_3),E_{n,n}(\mathcal{O}_{3}))$ } \\ \hline
$n$ & norms of elt. divisors & primes & rank \\ \hline
0&                                       && 0  \\ \hline
1& [ 3 ]                                 &(3)& 0\\ \hline
2& [ 3 ]                                 &(3)& 1\\ \hline
3& [ 3, 108 ]                            &(2,3)& 0\\ \hline
4& [ 3, 3, 12 ]                          &(2,3)& 0\\ \hline
5& [ 3, 3, 12 ]                          &(2,3)&  1\\ \hline
6& [ 3, 3, 3, 3, 10800 ]                 &(2,3,5)& 1 \\ \hline
7& [ 3, 3, 3, 3, 3, 12 ]                 &(2,3)& 1\\ \hline
8& [ 3, 3, 3, 3, 3, 12, 2352 ]           &(2,3,7)& 1\\ \hline
9& [ 3, 3, 3, 3, 3, 3, 3, 108, 972 ]     &(2,3)& 1 \\ \hline
\end{tabular}
\end{center}

\newpage
\section{Second cohomology} 

The main method I employ for computing the second cohomology is based on reduction theory as used in \cite{sv}. The cohomological dimension of Bianchi groups is 2 and 
the symmetric space they act on, namely the hyperbolic 3-space $\mathbb{H} \simeq \C \times \R^+$, is 3 dimensional. Reduction theory gives us a contractible 2 dimensional CW-complex inside $\mathbb{H}$ which is a deformation retract for the action of the Bianchi group. Moreover, the cellular action of the Bianchi group on the CW-complex is cocompact. This makes the CW-complex a suitable tool for cohomological computations. 

I will continue to focus on the Euclidean imaginary quadratic fields. The reduction theory for Bianchi groups has been worked out for these fields by Mendoza \cite{m} and Fl\"oge \cite{fl}. See also \cite{blw, fr}.

For an overview of Mendoza's construction, I refer readers to \cite{sv}. I will exhibit the method for the case of the Bianchi group $\Gamma=\psl_2(\Z[w])$ with $w=\sqrt{-2}$. 

Let $\mathcal{C}$ be the 2-dimensional CW-complex constructed by Mendoza for $\Gamma$. Then a fundamental cellular domain $\mathcal{F}$ for the action of $\Gamma$ on $\mathcal{C}$ is given by the area on the unit hemisphere centered at the origin of $\mathbb{H}$ above the rectangle in $\C \times \{ 0 \}$ with vertices $(\pm \frac{w}{2},0)$ and $(\frac{1}{2} \pm \frac{w}{2} ,0 )$.
 
Let 
$$a:=( \begin{smallmatrix} 1 & w  \\ w & -1 \\ \end{smallmatrix} ), \ \ b:=( \begin{smallmatrix} 1 & -1  \\ 1 & 0 \\ \end{smallmatrix} ), \ \  \ \ c:=( \begin{smallmatrix} 0 & -1  \\ 1 & 0 \\ \end{smallmatrix} ).$$

The stabilizers of the edges (1-cells) and the vertices (0-cells) of $\mathcal{F}$ are shown in the following picture \\
\vspace{.1 in}

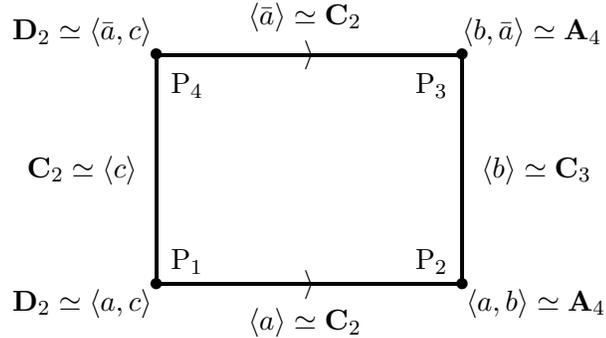
\begin{center}
\setlength{\unitlength}{.4in}
\begin{picture}(3,4)
\linethickness{1pt}
\put(0,0){\makebox(0,0){$\bullet$}}
\put(0.4,0.3){\makebox(0,0){P$_1$}}
\put(0,0){\line(0,1){3}}
\put(0,0){\line(1,0){4}}
\put(2,0){\makebox(0,0){$\rangle$}}
\put(0,3){\makebox(0,0){$\bullet$}}
\put(0.4,2.6){\makebox(0,0){P$_4$}}
\put(0,3){\line(1,0){4}}
\put(4,3){\makebox(0,0){$\bullet$}}
\put(3.6,2.6){\makebox(0,0){P$_3$}}
\put(4,0){\makebox(0,0){$\bullet$}}
\put(2,3){\makebox(0,0){$\rangle$}}
\put(3.6,0.3){\makebox(0,0){P$_2$}}
\put(4,0){\line(0,1){3}}
\put(-0.05,3.3){\makebox(0,0)[r]{${\bf D}_2 \simeq \langle \bar{a},c \rangle $}}
\put(1.2,3.50){\makebox(0,0)[s]{$\langle \bar{a} \rangle \simeq {\bf C}_2$}}
\put(4,3.3){\makebox(0,0)[s]{$\langle b,\bar{a} \rangle \simeq {\bf A}_4$}}
\put(1.2,-0.50){\makebox(0,0)[s]{$\langle a \rangle \simeq {\bf C}_2$}}
\put(4.25,1.5){\makebox(0,0)[l]{$\langle b \rangle \simeq {\bf C}_3$}}
\put(-0.25,1.5){\makebox(0,0)[r]{${\bf C}_2 \simeq \langle c \rangle $}}
\put(-0.05,-0.25){\makebox(0,0)[r]{${\bf D}_2 \simeq \langle a,c \rangle $}}
\put(4.05,-0.25){\makebox(0,0)[l]{$ \langle a,b \rangle \simeq {\bf A}_4$}}
\end{picture}
\end{center}
\vspace{.5 in}

The horizontal edges are identified by the element $g= ( \begin{smallmatrix} 1 & w  \\ 0 & 1 \\ \end{smallmatrix} )$, that is 
$g P_1P_2=P_4P_3$. Thus the quotient by $\Gamma$ is a cylinder. Moreover, the stabilizer of the whole rectangle (2-cell) is trivial. 

From this combinatorial data, one can compute the (co)homology. One way to do this is to feed the data into the equivariant cohomology spectral sequence
$$E^{p,q}_1(M)=\displaystyle{\bigoplus_{\sigma
\in \Sigma_p}} H^q(\Gamma_{\sigma}, M)\Longrightarrow
H^{p+q}(\Gamma, M).$$
\noindent where $M$ is any $\mathbb Z\Gamma$-module and $\Sigma_p$ is a set of
representatives of all the $\Gamma$-orbits of the $p$-cells of $\mathcal{C}$. See page 164 of \cite{b} for a description. The homological version of this spectral sequence has been used in $\cite{sv},\cite{fr}$. See Section 10 of \cite{yasaki} for another method to extract the same information.

Let $\Gamma_i, \Gamma_{ij}$ stand for the stabilizers of the vertex P$_i$ and the edge between P$_i$ and P$_j$ respectively. 
Let $M$ be a right $\Gamma$-module over $\Z[w][\frac{1}{6}]$. As primes above 2 and 3 are inverted, the cohomology of the (finite) stabilizers vanish in degree greater than 0. Hence, we have $E^{p,q}_1(M) = 0$ for all $q > 0$. Therefore, the
spectral sequence is concentrated on the horizontal axis $q = 0$ and the cohomology of the cochain complex
$$E^{0,0}_1 \buildrel{d^{0,0}_1}\over{\longrightarrow} E^{1,0}_1(M) \buildrel{d^{1,0}_1}\over{\longrightarrow} E^{2,0}_1(M) $$
gives $H^*(\Gamma,M)$, that is 
$$H^0(\Gamma, M)= Ker (d^{0,0}_1) \ \ H^1(\Gamma, M)= Ker (d^{1,0}_1)/Im (d^{0,0}_1),
\ \ H^2(\Gamma, M)= M/Im (d^{1,0}_1).$$
Now with the appropriate substitutions, the cochain complex reads
$$\bigoplus_{vertex \ i}H^0(\Gamma_i, M)\buildrel{d^{0,0}_1}\over{\longrightarrow} \bigoplus_{edge \ ij} H^0(\Gamma_{ij}, M)\buildrel{d^{1,0}_1}\over{\longrightarrow}H^0(\langle Id \rangle, M).$$
Here $\langle Id \rangle$ is the trivial stabilizer of the 2-cell $\mathcal{F}$.

To explicitly compute $H^2$, it remains to describe the differential $d^{1,0}_1$. One can choose the orientation on $\mathcal{F}$ so that the differential map becomes as follows
$$M^{\Gamma_1}\oplus M^{\Gamma_2} \buildrel{d^{0,0}_1}\over{\longrightarrow} M^{\Gamma_{12}}\oplus M^{\Gamma_{23}}\oplus M^{\Gamma_{41}} \buildrel{d^{1,0}_1}\over{\longrightarrow}M.$$

$$d_1^{1,0}(m_{12},m_{23},m_{41})=m_{12}+ m_{23}+m_{41}-m_{12}\cdot g^{-1}$$

The information on the fundamental 2-cell for the groups $\psl_2(\mathcal{O}_d)$ with $-d=1,2,3,7,11$ is included in the article \cite{sv}, so I do not repeat it here. The same information for the groups $\pgl_2(\mathcal{O}_d)$ with $-d=1,2,3,7,11$ is not included in that article and one needs to go to the above mentioned thesis of Mendoza (although a few of them are contained in \cite{blw} as well) which is very hard to access from outside of Germany. So I will now describe the information on these groups in the pictorial form as above. \\

\noindent $\bullet \ \ \pgl_2(\mathcal{O}_1)$ \\

Put $i=\sqrt{-1}$. Let 
$$a:=( \begin{smallmatrix} 0 & i  \\ i & 0 \\ \end{smallmatrix} ), \ \ b:=( \begin{smallmatrix} 1 & -1  \\ 1 & 0 \\ \end{smallmatrix} ), \ \  \ \ c:=( \begin{smallmatrix} 0 & i  \\ 1 & 0 \\ \end{smallmatrix} ).$$

The stabilizers of the edges and the vertices of $\mathcal{F}$ are shown in the following picture \\
\vspace{.1 in}

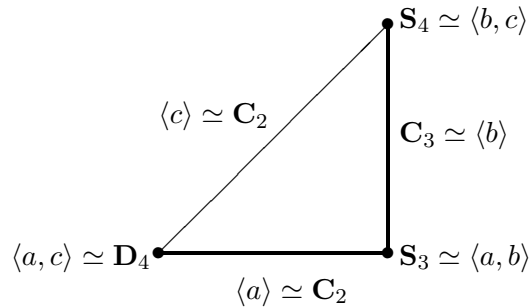
\begin{center}
\setlength{\unitlength}{.4in}
\begin{picture}(3,4)
\linethickness{1pt}
\put(0,0){\makebox(0,0){$\bullet$}}
\put(3,0){\makebox(0,0){$\bullet$}}
\put(3,3){\makebox(0,0){$\bullet$}}
\put(3,0){\line(0,1){3}}
\put(0,0){\line(1,0){3}}
\put(0,0){\line(1,1){3}}
\put(-0.1,-0.05){\makebox(0,0)[r]{$ \langle a,c \rangle \simeq {\bf D}_4  $}}
\put(3.15,-0.05){\makebox(0,0)[l]{$ {\bf S}_3 \simeq \langle a,b \rangle  $}}
\put(3.15, 1.55){\makebox(0,0)[l]{$ {\bf C}_3 \simeq \langle b \rangle  $}}
\put(3.15,3.05){\makebox(0,0)[l]{$ {\bf S}_4 \simeq \langle b,c \rangle   $}}
\put(1,-0.50){\makebox(0,0)[s]{$\langle a \rangle \simeq {\bf C}_2$}}
\put(0,1.8){\makebox(0,0)[s]{$\langle c \rangle \simeq {\bf C}_2$}}
\end{picture}
\end{center}
\vspace{.5 in}

There are no identifications and the stabilizer of the triangle (2-cell) is trivial. 
\vspace{.1 in}

\noindent $\bullet \ \ \pgl_2(\mathcal{O}_2)$ \\

Put $w=\sqrt{-2}$. Let 
$$a:=( \begin{smallmatrix} 0 & 1  \\ 1 & 0 \\ \end{smallmatrix} ), \ \ b:=( \begin{smallmatrix} 1 & -1  \\ 1 & 0 \\ \end{smallmatrix} ), \ \  \ \ c:=( \begin{smallmatrix} w & 1  \\ 1 & 0 \\ \end{smallmatrix} ), \ \ d:=( \begin{smallmatrix} 0 & -1  \\ 1 & 0 \\ \end{smallmatrix} ).$$

The stabilizers of the edges and the vertices of $\mathcal{F}$ are shown in the following picture \\
\vspace{.1 in}

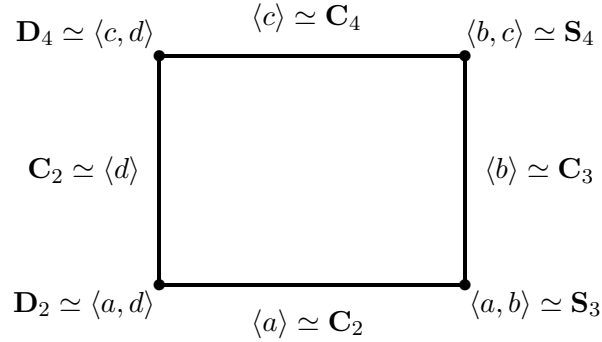
\begin{center}
\setlength{\unitlength}{.4in}
\begin{picture}(3,4)
\linethickness{1pt}
\put(0,0){\makebox(0,0){$\bullet$}}
\put(0,3){\makebox(0,0){$\bullet$}}
\put(4,3){\makebox(0,0){$\bullet$}}
\put(4,0){\makebox(0,0){$\bullet$}}
\put(0,0){\line(0,1){3}}
\put(0,0){\line(1,0){4}}
\put(0,3){\line(1,0){4}}
\put(4,0){\line(0,1){3}}
\put(-0.05,-0.25){\makebox(0,0)[r]{${\bf D}_2 \simeq \langle a,d \rangle $}}
\put(1.2,-0.50){\makebox(0,0)[s]{$\langle a \rangle \simeq {\bf C}_2$}}
\put(4.05,-0.25){\makebox(0,0)[l]{$ \langle a,b \rangle \simeq {\bf S}_3$}}
\put(4.25,1.5){\makebox(0,0)[l]{$\langle b \rangle \simeq {\bf C}_3$}}
\put(4,3.3){\makebox(0,0)[s]{$\langle b,c \rangle \simeq {\bf S}_4$}}
\put(1.2,3.50){\makebox(0,0)[s]{$\langle c \rangle \simeq {\bf C}_4$}}
\put(-0.05,3.3){\makebox(0,0)[r]{${\bf D}_4 \simeq \langle c,d \rangle $}}
\put(-0.25,1.5){\makebox(0,0)[r]{${\bf C}_2 \simeq \langle d \rangle $}}

\end{picture}
\end{center}
\vspace{.5 in}

There are no identifications and the stabilizer of the rectangle (2-cell) is trivial. 
\vspace{.1 in}

\noindent $\bullet \ \ \pgl_2(\mathcal{O}_3)$ \\

Put $w=\frac{1+\sqrt{-3}}{2}$. Let 
$$a:=( \begin{smallmatrix} 0 & 1 \\ 1 & 0 \\ \end{smallmatrix} ), \ \ b:=( \begin{smallmatrix} 1 & -1  \\ 1 & 0 \\ \end{smallmatrix} ), \ \  \ \ c:=( \begin{smallmatrix} 0 & w  \\ 1 & 0 \\ \end{smallmatrix} ).$$

The stabilizers of the edges and the vertices of $\mathcal{F}$ are shown in the following picture \\
\vspace{.1 in}

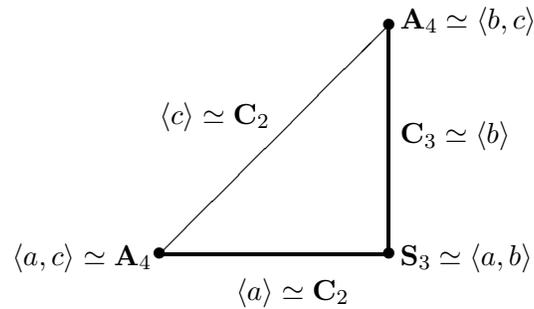
\begin{center}
\setlength{\unitlength}{.4in}
\begin{picture}(3,4)
\linethickness{1pt}
\put(0,0){\makebox(0,0){$\bullet$}}
\put(3,0){\makebox(0,0){$\bullet$}}
\put(3,3){\makebox(0,0){$\bullet$}}
\put(3,0){\line(0,1){3}}
\put(0,0){\line(1,0){3}}
\put(0,0){\line(1,1){3}}
\put(-0.1,-0.05){\makebox(0,0)[r]{$ \langle a,c \rangle \simeq {\bf A}_4  $}}
\put(1,-0.50){\makebox(0,0)[s]{$\langle a \rangle \simeq {\bf C}_2$}}
\put(3.15,-0.05){\makebox(0,0)[l]{$ {\bf S}_3 \simeq \langle a,b \rangle  $}}
\put(3.15, 1.55){\makebox(0,0)[l]{$ {\bf C}_3 \simeq \langle b \rangle  $}}
\put(3.15,3.05){\makebox(0,0)[l]{$ {\bf A}_4 \simeq \langle b,c \rangle   $}}
\put(0,1.8){\makebox(0,0)[s]{$\langle c \rangle \simeq {\bf C}_2$}}
\end{picture}
\end{center}
\vspace{.5 in}

There are no identifications and the stabilizer of the triangle (2-cell) is trivial. 
\vspace{.1 in}

\noindent $\bullet \ \ \pgl_2(\mathcal{O}_7)$ \\

Put $w=\frac{1+\sqrt{-7}}{2}$. Let 
$$a:=( \begin{smallmatrix} 0 & 1  \\ 1 & 0 \\ \end{smallmatrix} ), \ \ b:=( \begin{smallmatrix} 1 & -1  \\ 1 & 0 \\ \end{smallmatrix} ), \ \  \ \ c:=( \begin{smallmatrix} 1 & -w  \\ \overline{w} & -1 \\ \end{smallmatrix} ), \ \ d:=( \begin{smallmatrix} 0 & -1  \\ 1 & 0 \\ \end{smallmatrix} ).$$

The stabilizers of the edges and the vertices of $\mathcal{F}$ are shown in the following picture \\
\vspace{.1 in}

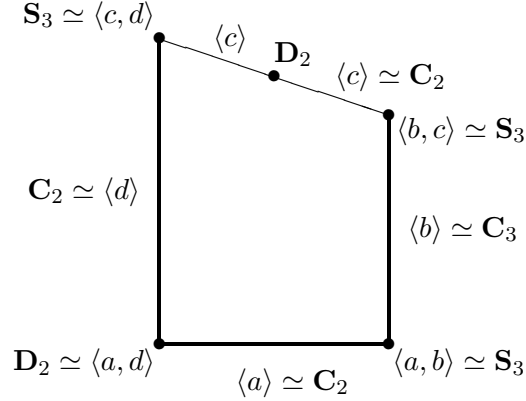
\begin{center}
\setlength{\unitlength}{.4in}
\begin{picture}(3,4)
\linethickness{1pt}
\put(0,0){\makebox(0,0){$\bullet$}}
\put(0,4){\makebox(0,0){$\bullet$}}
\put(3,3){\makebox(0,0){$\bullet$}}
\put(3,0){\makebox(0,0){$\bullet$}}
\put(1.5,3.5){\makebox(0,0){$\bullet$}}
\put(0,0){\line(0,1){4}}
\put(0,0){\line(1,0){3}}
\put(0,4){\line(3,-1){3}}
\put(3,0){\line(0,1){3}}
\put(-0.05,-0.25){\makebox(0,0)[r]{${\bf D}_2 \simeq \langle a,d \rangle $}}
\put(1,-0.50){\makebox(0,0)[s]{$\langle a \rangle \simeq {\bf C}_2$}}
\put(3.05,-0.25){\makebox(0,0)[l]{$ \langle a,b \rangle \simeq {\bf S}_3$}}
\put(3.25,1.5){\makebox(0,0)[l]{$\langle b \rangle \simeq {\bf C}_3$}}
\put(3.1,2.8){\makebox(0,0)[s]{$\langle b,c \rangle \simeq {\bf S}_3$}}
\put(2.3,3.5){\makebox(0,0)[s]{$\langle c \rangle \simeq {\bf C}_2$}}
\put(1.5,3.8){\makebox(0,0)[s]{${\bf D}_2$}}
\put(.7,4){\makebox(0,0)[s]{$\langle c \rangle$}}
\put(-0.05,4.3){\makebox(0,0)[r]{${\bf S}_3 \simeq \langle c,d \rangle $}}
\put(-0.25,2){\makebox(0,0)[r]{${\bf C}_2 \simeq \langle d \rangle $}}

\end{picture}
\end{center}
\vspace{.5 in}

The two adjacent short edges on the top are identified via $g= ( \begin{smallmatrix} 1 & -w  \\ 0 & -1 \\ \end{smallmatrix} )$ which fixes 
the vertex between them. Thus these two edges are oppositely oriented and the stabilizer of the vertex between them is ${\bf D}_2 \simeq \langle c, g \rangle$. Again the stabilizer of the whole 2-cell is trivial. 
\vspace{.1 in}

\noindent $\bullet \ \ \pgl_2(\mathcal{O}_{11})$ \\

Put $w=\frac{1+\sqrt{-11}}{2}$. Let 
$$a:=( \begin{smallmatrix} 0 & 1  \\ 1 & 0 \\ \end{smallmatrix} ), \ \ b:=( \begin{smallmatrix} 1 & -1  \\ 1 & 0 \\ \end{smallmatrix} ), \ \  \ \ c:=( \begin{smallmatrix} 1 & -w  \\ \overline{w} & -2 \\ \end{smallmatrix} ), \ \ d:=( \begin{smallmatrix} 0 & -1  \\ 1 & 0 \\ \end{smallmatrix} ).$$

The stabilizers of the edges and the vertices of $\mathcal{F}$ are shown in the following picture \\
\vspace{.1 in}

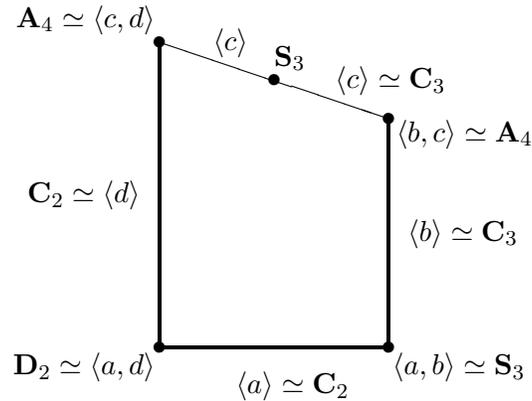
\begin{center}
\setlength{\unitlength}{.4in}
\begin{picture}(3,4)
\linethickness{1pt}
\put(0,0){\makebox(0,0){$\bullet$}}
\put(0,4){\makebox(0,0){$\bullet$}}
\put(3,3){\makebox(0,0){$\bullet$}}
\put(3,0){\makebox(0,0){$\bullet$}}
\put(1.5,3.5){\makebox(0,0){$\bullet$}}
\put(0,0){\line(0,1){4}}
\put(0,0){\line(1,0){3}}
\put(0,4){\line(3,-1){3}}
\put(3,0){\line(0,1){3}}
\put(-0.05,-0.25){\makebox(0,0)[r]{${\bf D}_2 \simeq \langle a,d \rangle $}}
\put(1,-0.50){\makebox(0,0)[s]{$\langle a \rangle \simeq {\bf C}_2$}}
\put(3.05,-0.25){\makebox(0,0)[l]{$ \langle a,b \rangle \simeq {\bf S}_3$}}
\put(3.25,1.5){\makebox(0,0)[l]{$\langle b \rangle \simeq {\bf C}_3$}}
\put(3.1,2.8){\makebox(0,0)[s]{$\langle b,c \rangle \simeq {\bf A}_4$}}
\put(2.3,3.5){\makebox(0,0)[s]{$\langle c \rangle \simeq {\bf C}_3$}}
\put(1.5,3.8){\makebox(0,0)[s]{${\bf S}_3$}}
\put(.7,4){\makebox(0,0)[s]{$\langle c \rangle$}}
\put(-0.05,4.3){\makebox(0,0)[r]{${\bf A}_4 \simeq \langle c,d \rangle $}}
\put(-0.25,2){\makebox(0,0)[r]{${\bf C}_2 \simeq \langle d \rangle $}}

\end{picture}
\end{center}
\vspace{.5 in}

The two adjacent short edges on the top are identified via $g= ( \begin{smallmatrix} 1 & -w  \\ 0 & -1 \\ \end{smallmatrix} )$ which fixes 
the vertex between them. Thus these two edges are oppositely oriented and the stabilizer of the vertex between them is ${\bf S}_3 \simeq \langle c, g \rangle$. Again the stabilizer of the whole 2-cell is trivial. 

\subsection{data on the integral second cohomology: level 1}

I have implemented the above algorithm in MAGMA. I have not inverted the primes above 2,3 and thus the computations may not give correct data on 2,3 torsion. Below I give a complete list of the primes that appear in the torsion part of the second cohomology of both $\psl$ and $\pgl$.
The large primes are highlighted in bold face. 

The data imply that if $(p)$ ramifies in $\mathcal{O}$ then there is 
$p$-torsion in the integral second cohomology (except the case $k=\ell=0$) but I have not been able to prove that this is always the case.

\begin{center} 
\begin{tabular}{|c|c|c|} 
\multicolumn{3}{c}{$H^2(\psl_2(\mathcal{O}_1),E_{n,n}(\mathcal{O}_1))$} \\ \hline
$n$ & primes & rank \\ \hline
1&  [ ]     & 1   \\ \hline
2& [ 2 ]     & 1  \\ \hline
3& [ 2, 3]   & 1  \\ \hline
4& [ 2, 3 ]  & 1  \\ \hline
5& [ 2 ]     & 2  \\ \hline
6& [ 2, 3, 5 ]      & 1  \\ \hline
7& [ 2, 3, 7 ]      & 3  \\ \hline
8& [ 2, 3, 5, 7 ]   & 1  \\ \hline
9& [ 2, 3 ]         & 3  \\ \hline
10& [ 2, 3, 5, 7 ]  & 2  \\ \hline
11& [ 2, 3, 5, 11 ] & 4  \\ \hline
12& [ 2, 3, 5, 7, 11 ]         & 1 \\ \hline
13& [ 2, 3, 5 ]                & 5 \\ \hline
14& [ 2, 3, 5, 7, 11, 13 ]     & 2 \\ \hline
15& [ 2, 3, 5, 7 ]             & 5 \\ \hline
16& [ 2, 3, 5, 7, 11, 13 ]     & 2  \\ \hline
17& [ 2, 3, 5, 7 ]             & 6 \\ \hline
18& [ 2, 3, 5, 7, 11, 13, 17, {\bf 19}, {\bf 23} ]          & 2   \\ \hline     
19& [ 2, 3, 5, 7, 13, 19 ]                      & 7  \\ \hline
20& [ 2, 3, 5, 7, 11, 13, 17, 19, {\bf 409}, {\bf 6997} ]   & 2   \\ \hline
21& [ 2, 3, 5, 7, {\bf 59} ]                          & 7   \\ \hline         
22& [ 2, 3, 5, 7, 11, 13, 17, 19, {\bf 13707791} ]    & 3   \\ \hline
23& [ 2, 3, 5, 7, 11, {\bf 23}, {\bf 113} ]                 & 8   \\ \hline     
24& [ 2, 3, 5, 7, 11, 13, 17, 19, 23, {\bf 1033}, {\bf 4457}, {\bf 18743} ] & 2 \\ \hline
25& [ 2, 3, 5, 7, 11, 13, 17, {\bf 1523} ] & 9 \\ \hline                           
\end{tabular}
\end{center}

\begin{center} 
\begin{tabular}{|c|c|c|} 
\multicolumn{3}{c}{$H^2(\pgl_2(\mathcal{O}_1),E_{n,n}(\mathcal{O}_1))$} \\ \hline
$n$ & primes & rank \\ \hline
1& [ ] &1 \\ \hline
2& [ 2]&1 \\ \hline
3& [ 2, 3]&1 \\ \hline
4& [ 2 ]&1 \\ \hline
5& [ 2 ]&2 \\ \hline
6& [ 2, 3, 5 ]&1 \\ \hline
7& [ 2, 3, 7 ]&2 \\ \hline
8& [ 2, 3, 5, 7 ]&1 \\ \hline
9& [ 2, 3 ]&3 \\ \hline
10& [ 2, 3, 5, 7 ]&2 \\ \hline
11& [ 2, 3, 11 ]&3 \\ \hline
12& [ 2, 3, 5, 11 ]&1 \\ \hline
13& [ 2, 3, 5 ]&4 \\ \hline
14& [ 2, 3, 5, 7, 11, 13 ]&2 \\ \hline
15& [ 2, 3, 5, 7 ]&4 \\ \hline
16& [ 2, 3, 5, 7, 13 ]&2 \\ \hline
17& [ 2, 3, 5, 7 ]&5 \\ \hline
18& [ 2, 3, 5, 7, 11, 17 ]&2 \\ \hline
19& [ 2, 3, 5, 19 ]&5 \\ \hline
20& [ 2, 3, 5, 7, 13, 17, 19, {\bf 409}]&  2 \\ \hline   
21& [ 2, 3, 5, 7 ]&6 \\ \hline
22& [ 2, 3, 5, 7, 11, 19 ]&3 \\ \hline
23& [ 2, 3, 5, 7, 11, 23 ]&6 \\ \hline
24& [ 2, 3, 5, 7, 11, 13, 17, 23, {\bf 1033} ]&  2 \\ \hline   
25& [ 2, 3, 5, 7, 11, 17 ]&7 \\ \hline     
26& [ 2, 3, 5, 7, 11, 13, 19, 23, {\bf 157}, {\bf 683} ]  &  3 \\ \hline 
27& [ 2, 3, 5, 7 ] &  7 \\ \hline 
28& [ 2, 3, 5, 7, 11, 13, 17, {\bf 664197637} ]    &  3 \\ \hline      
29& [ 2, 3, 5, 7, 11, 13, {\bf 89} ]         &  8 \\ \hline           
30& [ 2, 3, 5, 7, 11, 13, 19, 23, 29, {\bf 211}, {\bf 36312691} ]   &  3 \\ \hline  
\end{tabular}
\end{center}

\vspace{1 in}

\begin{center} 
\begin{tabular}{|c|c|c|} 
\multicolumn{3}{c}{$H^2(\psl_2(\mathcal{O}_2),E_{n,n}(\mathcal{O}_2))$} \\ \hline
$n$ & primes & rank \\ \hline
1& [  ]& 1  \\ \hline
2& [ 2 ]&1 \\ \hline
3& [ 2, 3 ]&2 \\ \hline
4& [ 2, 3 ]&1 \\ \hline
5& [ 2, 3, 5 ]&3 \\ \hline
6& [ 2, 3, 5 ]&2 \\ \hline
7& [ 2, 3, 5, 7 ]&4 \\ \hline
8& [ 2, 3, 5, 7 ]&2 \\ \hline
9& [ 2, 3, 5, 7, {\bf 31} ] &  5    \\ \hline                  
10& [ 2, 3, 5, 7 ]&3 \\ \hline
11& [ 2, 3, 5, 7, 11 ]&6 \\ \hline
12& [ 2, 3, 5, 7, 11, {\bf 37} ]  &  3   \\ \hline             
13& [ 2, 3, 5, 7, 11, 13, {\bf 547} ] &  7         \\ \hline   
14& [ 2, 3, 5, 7, 11, 13, {\bf 439}, {\bf 110281} ] &  4    \\ \hline
15& [ 2, 3, 5, 7, 11, 13, {\bf 61}, {\bf 163} ] &   8       \\ \hline
\end{tabular}
\end{center}

\vspace{.5 in}

\begin{center} 
\begin{tabular}{|c|c|c|} 
\multicolumn{3}{c}{$H^2(\pgl_2(\mathcal{O}_2),E_{n,n}(\mathcal{O}_2))$} \\ \hline
$n$ & primes & rank \\ \hline
1& [  ]& 1  \\ \hline
2& [ 2 ]& 1 \\ \hline
3& [ 2 ]& 2 \\ \hline
4& [ 2, 3 ]&1 \\ \hline
5& [ 2, 5 ]&3 \\ \hline
6& [ 2, 3, 5 ]&1 \\ \hline
7& [ 2, 3, 7 ]&4 \\ \hline
8& [ 2, 3, 5, 7 ]&1 \\ \hline
9& [ 2, 3 ]&5 \\ \hline
10& [ 2, 3, 5, 7 ]&2 \\ \hline
11& [ 2, 3, 5 ]&6 \\ \hline
12& [ 2, 3, 5, 7, 11, {\bf 37} ] &  1 \\ \hline                                 
13& [ 2, 3, 5, 13 ]&7 \\ \hline
14& [ 2, 3, 5, 7, 11, 13, {\bf 110281}] & 2 \\ \hline                          
15& [ 2, 3, 5, 7 ]&8 \\ \hline
16& [ 2, 3, 5, 7, 11, 13, {\bf 1671337} ] &   2 \\ \hline                       
17& [ 2, 3, 5, 7, {\bf 103} ]   &     9   \\ \hline                             
18& [ 2, 3, 5, 7, 11, 13, 17, {\bf 3812807473} ]  &    2  \\ \hline             
19& [ 2, 3, 5, 7, {\bf 907} ]     &    10   \\ \hline                           
20& [ 2, 3, 5, 7, 11, 13, 17, 19, {\bf 3511}, {\bf 879556698451244053}] & 2  \\ \hline
\end{tabular}
\end{center}

\vspace{1 in}

\begin{center} 
\begin{tabular}{|c|c|c|} 
\multicolumn{3}{c}{$H^2(\psl_2(\mathcal{O}_3),E_{n,n}(\mathcal{O}_3))$} \\ \hline
$n$ & primes & rank  \\ \hline
1& [ ] & 1   \\ \hline
2& [ 3 ]& 1   \\ \hline
3& [ 3 ]& 1   \\ \hline
4& [ 2, 3 ]& 1   \\ \hline
5& [ 2, 3, 5 ]& 1  \\ \hline
6& [ 2, 3, 5 ]& 2  \\ \hline
7& [ 2, 3, 5 ]& 2  \\ \hline
8& [ 2, 3, 7 ]& 1  \\ \hline
9& [ 2, 3, 7 ]& 2  \\ \hline
10& [ 2, 3, 5 ]& 3  \\ \hline
11& [ 2, 3, 5, 11 ]& 2  \\ \hline
12& [ 2, 3, 5, 7, 11 ]& 2  \\ \hline
13& [ 2, 3, 5, 11 ]& 3  \\ \hline
14& [ 2, 3, 5, 7, 13 ]& 3  \\ \hline
15& [ 2, 3, 5, 7, 13 ]& 3  \\ \hline
16& [ 2, 3, 5, 7, 11 ]& 3  \\ \hline
17& [ 2, 3, 5, 7, 17 ]& 3  \\ \hline
18& [ 2, 3, 5, 7, 13, 17 ]& 4  \\ \hline
19& [ 2, 3, 5, 7, 11, 17, {\bf 61} ]& 4         \\ \hline        
20& [ 2, 3, 5, 7, 19 ]& 3  \\ \hline
21& [ 2, 3, 5, 7, 13, 19, {\bf 151} ]& 4    \\ \hline            
22& [ 2, 3, 5, 7, 11, 17 ]& 5  \\ \hline
23& [ 2, 3, 5, 7, 11, 23, {\bf 103} ]& 4        \\ \hline        
24& [ 2, 3, 5, 7, 11, 13, 17, 19, 23, {\bf 53} ]& 4     \\ \hline
25& [ 2, 3, 5, 7, 11, 17, 23, {\bf 29}, {\bf 947} ] &  5     \\ \hline
26& [ 2, 3, 5, 7, 11, 13 ]& 5  \\ \hline
\end{tabular}
\end{center}

\begin{center} 
\begin{tabular}{|c|c|c|} 
\multicolumn{3}{c}{$H^2(\pgl_2(\mathcal{O}_3),E_{n,n}(\mathcal{O}_3))$} \\ \hline
$n$ & primes & rank  \\ \hline
1& [ ]& 1  \\ \hline
2& [ ]& 1  \\ \hline
3& [ 2, 3 ]& 1  \\ \hline
4& [ 2, 3 ]&1  \\ \hline
5& [ 2, 3, 5 ]&1  \\ \hline
6& [ 2, 3, 5 ]&1  \\ \hline
7& [ 2, 3 ]&2  \\ \hline
8& [ 2, 3, 7 ]&1  \\ \hline
9& [ 2, 3 ]&2  \\ \hline
10& [ 2, 3, 5 ]&2  \\ \hline
11& [ 2, 3, 5, 11 ]&2  \\ \hline
12& [ 2, 3, 5, 7, 11 ]&1  \\ \hline
13& [ 2, 3 ]&3  \\ \hline
14& [ 2, 3, 5, 7, 13 ]&2  \\ \hline
15& [ 2, 3, 5 ]&3  \\ \hline
16& [ 2, 3, 5, 7, 11 ]&2  \\ \hline
17& [ 2, 3, 5, 17 ]&3  \\ \hline
18& [ 2, 3, 5, 7, 13, 17 ]&2  \\ \hline
19& [ 2, 3, 5, 7 ]&4  \\ \hline
20& [ 2, 3, 5, 7, 19 ]&2  \\ \hline
21& [ 2, 3, 5, 7 ]&4  \\ \hline
22& [ 2, 3, 5, 7, 11, 17 ]&3  \\ \hline
23& [ 2, 3, 5, 11, 23 ]&4  \\ \hline
24& [ 2, 3, 5, 7, 11, 13, 19, 23, {\bf 53} ] &  2  \\ \hline     
25& [ 2, 3, 5, 7 ] & 5  \\ \hline
26& [ 2, 3, 5, 7, 11, 13 ] & 3 \\ \hline
27& [ 2, 3, 5, 7, 11 ]  & 5 \\ \hline
28& [ 2, 3, 5, 7, 11, 13, 17, 23 ] & 3 \\ \hline
29& [ 2, 3, 5, 7, 29 ] &  5 \\ \hline

\end{tabular}
\end{center}

\vspace{1 in}

\begin{center} 
\begin{tabular}{|c|c|c|} 
\multicolumn{3}{c}{$H^2(\psl_2(\mathcal{O}_7),E_{n,n}(\mathcal{O}_7))$} \\ \hline
$n$ & primes & rank \\ \hline
1& [ ]& 1  \\ \hline
2& [ 2, 7 ]&1  \\ \hline
3& [ 2, 3, 7 ]&1 \\ \hline
4& [ 2, 3, 7 ]&2 \\ \hline
5& [ 2, 3, 5, 7 ]&2 \\ \hline
6& [ 2, 3, 5, 7 ]&2 \\ \hline
7& [ 2, 3, 5, 7 ]&3 \\ \hline
8& [ 2, 3, 5, 7 ]&3 \\ \hline
9& [ 2, 3, 5, 7 ]&3 \\ \hline
10& [ 2, 3, 5, 7 ]&4 \\ \hline
11& [ 2, 3, 5, 7, 11 ]&4 \\ \hline
12& [ 2, 3, 5, 7, 11, {\bf 127} ]&6 \\ \hline
13& [ 2, 3, 5, 7, 11, 13, {\bf 31} ]&5 \\ \hline
14& [ 2, 3, 5, 7, 11, 13, {\bf 73} ]&5 \\ \hline
15& [ 2, 3, 5, 7, 11, 13, {\bf 271}, {\bf 431} ]&5 \\ \hline
16& [ 2, 3, 5, 7, 11, 13 ]&6 \\ \hline
17& [ 2, 3, 5, 7, 11, 13, 17, {\bf 37}, {\bf 67}, {\bf 89}, {\bf 101}, {\bf 277} ]&6 \\ \hline
18& [ 2, 3, 5, 7, 11, 13, 17, {\bf 43}, {\bf 457}, {\bf 2069}, {\bf 3323} ]&6 \\ \hline
\end{tabular}
\end{center}

\begin{center} 
\begin{tabular}{|c|c|c|} 
\multicolumn{3}{c}{$H^2(\pgl_2(\mathcal{O}_7),E_{n,n}(\mathcal{O}_7))$} \\ \hline
$n$ & primes & rank \\ \hline
1& [ ]&1 \\ \hline
2& [ 2 ]&1 \\ \hline
3& [ 2, 3, 7 ]&1 \\ \hline
4& [ 2, 3, 7 ]&1 \\ \hline
5& [ 2, 5, 7 ]&2 \\ \hline
6& [ 2, 3, 5, 7 ]&1 \\ \hline
7& [ 2, 3, 7 ]&3 \\ \hline
8& [ 2, 3, 5, 7 ]&1 \\ \hline
9& [ 2, 3, 7 ]&3 \\ \hline
10& [ 2, 3, 5, 7 ]&2 \\ \hline
11& [ 2, 3, 5, 7 ]&4 \\ \hline
12& [ 2, 3, 5, 7, 11, {\bf 127} ]&1 \\ \hline
13& [ 2, 3, 5, 7, 13 ]&5 \\ \hline
14& [ 2, 3, 5, 7, 11, 13, {\bf 73} ]&2 \\ \hline
15& [ 2, 3, 5, 7, {\bf 431} ]&5 \\ \hline
16& [ 2, 3, 5, 7, 11, 13 ]&2 \\ \hline
17& [ 2, 3, 5, 7, 17, {\bf 37} ]&6 \\ \hline
18& [ 2, 3, 5, 7, 11, 13, 17, {\bf 43}, {\bf 457}, {\bf 2069}, {\bf 3323} ]&2 \\ \hline
19& [ 2, 3, 5, 7, 13, 19, {\bf 311} ]&7 \\ \hline
20& [ 2, 3, 5, 7, 11, 13, 17, 19, {\bf 42197}, {\bf 12272815271} ]&2 \\ \hline
\end{tabular}
\end{center}

\begin{center} 
\begin{tabular}{|c|c|c|} 
\multicolumn{3}{c}{$H^2(\psl_2(\mathcal{O}_{11}),E_{n,n}(\mathcal{O}_{11}))$} \\ \hline
$n$ & primes & rank \\ \hline
1& [ ]&1 \\ \hline
2& [ 2 ]&2 \\ \hline
3& [ 2, 3, 11 ]&2 \\ \hline
4& [ 2, 3, 11 ]&2 \\ \hline
5& [ 2, 3, 5, 11 ]&3 \\ \hline
6& [ 2, 3, 5, 11 ]&4 \\ \hline
7& [ 2, 3, 5, 7, 11 ]&4 \\ \hline
8& [ 2, 3, 5, 7, 11 ]&4 \\ \hline
9& [ 2, 3, 5, 7, 11, {\bf 23} ]&5 \\ \hline
10& [ 2, 3, 5, 7, 11 ]&8 \\ \hline
11& [ 2, 3, 5, 7, 11, {\bf 37} ]&6 \\ \hline
12& [ 2, 3, 5, 7, 11 ]&6 \\ \hline
13& [ 2, 3, 5, 7, 11, 13, {\bf 43}, {\bf 19973} ]&7 \\ \hline
14& [ 2, 3, 5, 7, 11, 13 ]&8 \\ \hline
15& [ 2, 3, 5, 7, 11, 13, {\bf 31}, {\bf 47}, {\bf 1409}, {\bf 30817} ]&8 \\ \hline
16& [ 2, 3, 5, 7, 11, 13, 17, 19, {\bf 41}, {\bf 281} ]&8 \\ \hline
\end{tabular}
\end{center}

\vspace{1 in}
\begin{center} 
\begin{tabular}{|c|c|c|} 
\multicolumn{3}{c}{$H^2(\pgl_2(\mathcal{O}_{11}),E_{n,n}(\mathcal{O}_{11}))$} \\ \hline
$n$ & primes & rank \\ \hline
1& [ ]&1 \\ \hline
2& [ 2 ]&1 \\ \hline
3& [ 2 ]&2 \\ \hline
4& [ 2, 3, 11 ]&1 \\ \hline
5& [ 2, 11 ]&3 \\ \hline
6& [ 2, 3, 5, 11 ]&1 \\ \hline
7& [ 2, 3, 7, 11 ]&4 \\ \hline
8& [ 2, 3, 5, 7, 11 ]&1 \\ \hline
9& [ 2, 3, 11 ]&5 \\ \hline
10& [ 2, 3, 5, 7, 11 ]&4 \\ \hline
11& [ 2, 3, 5, 11 ]&6 \\ \hline
12& [ 2, 3, 5, 7, 11 ]&1 \\ \hline
13& [ 2, 3, 5, 11, 13 ]&7 \\ \hline
14& [ 2, 3, 5, 7, 11, 13 ]&2 \\ \hline
15& [ 2, 3, 5, 7, 11, {\bf 47} ]&8 \\ \hline
16& [ 2, 3, 5, 7, 11, 13, {\bf 41}, {\bf 281} ]&2 \\ \hline
17& [ 2, 3, 5, 7, 11, 17, {\bf 67} ]&9 \\ \hline 
18& [ 2, 3, 5, 7, 11, 13, 17, {\bf 449}, {\bf 20147}, {\bf 201797} ]&2 \\ \hline

\end{tabular}
\end{center}

\subsection{Congruence subgroups}
 Now let us focus on the second cohomology of congruence subgroups. For computational considerations, we will focus on the subgroups of the type
  $$ \Gamma_0(\mathfrak{a}): \Bigl \lbrace \bigl ( \begin{smallmatrix} a & b  \\ c & d \\ \end{smallmatrix} \bigr ) \in G : c \equiv 0 \mod \mathfrak{a} \Bigr \rbrace .$$
Here $G$ is the Bianchi group $\psl_2(\mathcal{O}_d)$ and $\mathfrak{a}$ is an ideal of $\mathcal{O}_d$, which is called the level.

\subsubsection{Trivial weight: torsion}

To compute the second cohomology with trivial weight, the approach employed in \cite{egm} is more efficient than the reduction theory approach used above. The idea is to compute the abelianization of the congruence subgroup $\Gamma$ using a (finite) presentation for the Bianchi group $G$ and the knowledge of the permutation action of the generators of $G$ on a set of 
coset representatives of $\Gamma$ in $G$. 

The relationship between first homology and the second cohomology is given by the Lefschetz duality. Let $R$ be any module in which 6 is invertible. Then for any $R[\Gamma]-$module $E$, we have
$$H_1(\Gamma, E) \simeq H^2_c(\Gamma,E)$$
where the right hand side is the cohomology with compact support, see Section 2. 

We are only interested in the case of the trivial coefficients $R=\Z[1/6]$. We need to study the exact sequence
$$H^1(\Gamma,R) \rightarrow H^1(U(\Gamma),R) \rightarrow H^2_c(\Gamma,R) \rightarrow H^2_{cusp}(\Gamma,R) \rightarrow 0.$$

Assume that $K$ has class number one and that $\Gamma = \Gamma_0(\p)$ where $\p$ is a prime ideal of residue degree one. It is easy to see that $\Gamma$ has only two cusps; $\{ 0,\infty \}$. That is $\abs{\Gamma \backslash \P}=2$ and the two classes are represented by the elements $( \begin{smallmatrix} 1 & 0 \\ 0 & 1 \\ \end{smallmatrix} ), ( \begin{smallmatrix} 0 & -1  \\ 1 & 0 \\ \end{smallmatrix} )$. 
If $-d \not = 1,3$, then the stabilizers of the cusps are free abelian of rank 2 and by Corollaire 3 of \cite{serre70} p.517, the image of 
$H^1(\Gamma,R)$ has rank 2 in $H^1(U(\Gamma),R)$ which has rank 4. For $-d=1,3$, the situation is complicated by the existence of torsion in the stabilizers of cusps. In this case, let $\Gamma_c^+:= \Gamma \cap U_{D_c}$ for each cusp $c$ of $\Gamma$ in the terminology of Section 2.1.
 Then $\Gamma_c / \Gamma_c^+ \simeq \mu := \langle ( \begin{smallmatrix} \varepsilon & 0 \\ 0 & \varepsilon^{-1} \\ \end{smallmatrix} )  \rangle $ where 
 $\varepsilon$ is a generator of the roots of unity in $\mathcal{O}_d$. The inflation-restriction sequence then gives 
 $$H^1(\Gamma_c,R) \simeq H^0( \mu , H^1(\Gamma_c^+,R)).$$
Now let us directly show that the latter is trivial. Without loss of generality, assume that $c=\infty$ and thus that 
$\Gamma^+_c := ( \begin{smallmatrix} 1 & *  \\ 0 & 1 \\ \end{smallmatrix} ) \cap \Gamma$. The action of $\mu$ on $H^1(\Gamma_c^+,R)$ is given as $(\tau f)(x)= f(\tau x \tau^{-1})\tau=f(\tau x \tau^{-1})$ for every $\tau \in \mu$ and every 1-cocyle $f : \Gamma_c^+ \rightarrow R$. Take $\varepsilon \in \mathcal{O}_d^*$ such that $\varepsilon^2=-1$ and put $\tau=( \begin{smallmatrix} \varepsilon & 0 \\ 0 & \varepsilon^{-1} \\ \end{smallmatrix} )$. Then for any element $x \in \Gamma_c^+$, $\tau x \tau^{-1}=x^{-1}$. The condition $f(x)=f(x^{-1})$ 
for every $x$ forces the 1-cocyle $f$ to be the trivial cocyle $f=0$. We are done.  
An alternative way is to use the geometric approach which I mostly avoided in this paper. Then our discussion here follows from the observation that the cross section of a cusp is an orbifold with underlying manifold a sphere in the cases $-d = 1,3$ and  is a torus in the remaining cases. 

I wrote programs to compute $\Gamma_0(\p)^{ab}$ in MAGMA for the Euclidean $\mathcal{O}_d$, that is, for $-d=1,2,3,7,11$. It is easily seen that the torsion get very large very quickly compared to the norm of the level of the congruence subgroup. Recall from the introduction that Grunewald and Schwermer speculated in \cite{gs} that no $p$-torsion appearing in abelianization of a finite index subgroup $\Gamma$ is greater then half of the index of $\Gamma$ inside the Bianchi group. In the case of 
$\Gamma_0(\mathfrak{p})$ , their speculation says that any $p$-torsion that appear in $\Gamma_{0}(\mathfrak{p})^{ab}$ should satisfy 
$$p \leq \dfrac{\textbf{N}\mathfrak{p}+1}{2}$$
where $\textbf{N}\mathfrak{p}$ is the norm of the prime ideal $\mathfrak{p}$.

The smallest primes that witness the falseness of this speculation for our five Bianchi groups $\psl_2(\mathcal{O}_d)$ are listed below in Table 
\ref{table: counter}.

\begin{table}[h]
\centering
\begin{tabular}{|c|c|c|c|}
 \hline
$d$ & Norm of $\mathfrak{p}$ & rank of $\Gamma_0(\mathfrak{p})^{ab}$ & prime torsion of $\Gamma_0(\mathfrak{p})^{ab}$ \\ \hline
1 & 401 & 0 & [2, 5, 41, 271 ] \\ \hline
2 & 193 & 2 & [2, 3, 23, 251 ] \\ \hline
3 & 937 & 0 & [2, 3, 13, 599 ] \\ \hline
7 & 137 & 2 & [2, 17, 83 ] \\ \hline
11 & 103 & 2 & [2, 3, 17, 19, 71 ] \\ \hline
\end{tabular}
\caption{smallest counter examples to Grunewald-Schwermer}
\label{table: counter}
\end{table}

My computations of $\Gamma_0(\mathfrak{p})^{ab}$ agree perfectly with that of \cite{egm} mentioned in the introduction in the part they overlap. It is easily observed that the primes in the torsion get to astronomical sizes even within the range $\textbf{N}\mathfrak{p} \leq 5000$. Below in Table 
\ref{table: sample} is a sample of the primes that appear in the torsion of $\Gamma_0(\mathfrak{p})^{ab}$ with $ 4900 \leq \textbf{N}\mathfrak{p} \leq 5000$ for the five Euclidean imaginary quadratic fields. The complete list is available on my website. 
It is interesting to observe that in the cases $-d=1,3$, where we have other units besides $\pm 1$, the torsion grows much slower than the other three cases. 

\begin{table} \footnotesize
\centering
\begin{tabular}{|c|c|} \hline
$\textbf{N}\mathfrak{p}$ &  {\bf some} of the primes that appear in the torsion of $\Gamma_0(\mathfrak{p})^{ab}$ \\ \hline
\multicolumn{2}{|c|}{$\psl(\mathcal{O}_1)$} \\ \hline
4909 & [ 2, 3, 7, 13, 409, 10691, 22871, 29423, 56980673, 71143433 ] \\ \hline
4933 & [ 2, 3, 37, 101, 137, 577, 947, 21169, 194981 ] \\ \hline
4937 & [ 2, 7, 37, 617, 10859, 108893, 4408530403, 157824962047 ] \\ \hline
4957 & [ 2, 3, 7, 13, 31, 59, 14347, 3051863, 9405667, 23132267 ] \\ \hline
4969 & [ 2, 3, 23, 71, 373, 191299, 39861006443, 8672729371087 ] \\ \hline
4973 & [ 2, 11, 13, 47, 71, 113, 127, 331, 6317, 7949, 39023, 628801, 2995319 ] \\ \hline
4993 & [ 2, 3, 5, 7, 11, 13, 101, 173, 798569, 5995036891, 18513420749 ] \\ \hline
\multicolumn{2}{|c|}{$\psl(\mathcal{O}_2)$} \\ \hline

4931 & [ $\hdots$ , 3772418780827, 67462419379713541, 442541106225737082232052179 ] \\ \hline
4937 & [ $\hdots$ , 1889149903, 7397090738497, 880941232181841675673769 ] \\ \hline
4969 & [ $\hdots$ , 2728733329370698225919458399, 114525595847400940348788195788260381871 ] \\ \hline
4987 & [ $\hdots$ , 1354882997352809, 167973141926075800477, 109210638303577813415629 ] \\ \hline
4993 & [ $\hdots$ , 15997185593, 14633678967206157243930187, 4844017554743814674462620193 ] \\ \hline

\multicolumn{2}{|c|}{$\psl(\mathcal{O}_3)$} \\ \hline
4903 & [ 3, 7, 19, 29, 37, 43, 61, 137, 191, 733 ] \\ \hline
4909 & [ 2, 3, 7, 13, 19, 47, 67, 409, 1409 ] \\ \hline
4933 & [ 2, 3, 5, 137, 173, 383, 719, 1451, 100057 ] \\ \hline
4951 & [ 3, 5, 7, 11, 271, 3797, 6696049 ] \\ \hline
4957 & [ 2, 3, 5, 7, 23, 43, 59, 233, 823, 62207 ] \\ \hline
4969 & [ 2, 3, 5, 7, 23, 181, 2591, 516336433 ] \\ \hline
4987 & [ 2, 3, 11, 71, 277, 619, 21977, 1971691 ] \\ \hline
4993 & [ 2, 3, 11, 13, 29, 727, 4153, 27127 ] \\ \hline
4999 & [ 2, 3, 7, 17, 29, 41, 83, 38593, 179623 ] \\ \hline

\multicolumn{2}{|c|}{$\psl(\mathcal{O}_7)$} \\ \hline

4909 & [ $\hdots$ , 3354447021713, 666100957349057134013, 13363557375430202095093 ] \\ \hline
4937 & [ $\hdots$ , 836083247742263, 60001748772648369971, 1344885261548364695671 ] \\ \hline
4943 & [ $\hdots$ , 94861335404089, 157213239530981, 345644733766517, 714087340201211 ] \\ \hline
4951 & [ $\hdots$ , 42137202713,11756096619570265637, 47745831545933513537 ] \\ \hline
4957 & [ $\hdots$ , 6803766726937001299, 21088956680308937473, 34130091188757085391 ] \\ \hline
4967 & [ $\hdots$ , 42061245937, 3414861551033731, 385786872173747641 ] \\ \hline
4993 & [ $\hdots$ , 16112554517, 22230923149, 47405513059, 17179435084786759 ] \\ \hline
4999 & [ $\hdots$ , 47183940647, 47747826462797, 176725513764138170761817312541116531 ] \\ \hline

\multicolumn{2}{|c|}{$\psl(\mathcal{O}_{11})$} \\ \hline

4909 & [ $\hdots$ , 491602700153184794115037, 3160753948740219890398523741106925031 ] \\ \hline
4931 & [ $\hdots$ , 59242366654994144915737, 397153057377536493107457514082773 ] \\ \hline
4933 & [ $\hdots$ , 471591580131222099301009, 753357254439534230416253 ] \\ \hline
4937 & [ $\hdots$ , 774606120056702384410790118960699805738139 ] \\ \hline
4943 & [ $\hdots$ , 49685906201385872741, 7533150099701393721041, 1806172579157695730540919793 ] \\ \hline
4951 & [ $\hdots$ , 32561299447966536475490232836221, 575858582707156517384453334853901 ] \\ \hline
4973 & [ $\hdots$ , 668079334182971453623, 2223356120717452698676440064717 ] \\ \hline
4987 & [ $\hdots$ , 26685596532560442049106969671, 121708009502005164710374726093 ] \\ \hline
4999 & [ $\hdots$ , 35270997998154652004835942597708494620078410433635847 ] \\ \hline
\end{tabular}
\caption{ a modest sample of large torsion occurring in $\Gamma_0(\mathfrak{p})^{ab}$}
\label{table: sample}
\end{table}

\subsubsection{Trivial weight: rank}

 In this section I will report on the rank of $H^2_{cusp}(\Gamma_0(\mathfrak{a}),\mathcal{O})$. This rank is clearly equal to the dimension of 
$H^2_{cusp}(\Gamma_0(\mathfrak{a}),\C)$ and thus its nonvanishing is conjecturally connected to abelian varieties of GL$_2-$type over imaginary quadratic fields (see \cite{crem, egm, sen_hyper}). Moreover, the vanishing of this rank in certain cases is equivalent to the existence of rational homology spheres (see \cite{lmr}). In this section I will report on my computations related to these two aspects.
 
 As explained in the previous subsection, the rank of $\Gamma_0(\mathfrak{a})^{ab}$ is related to the rank $r$ of $H^2_{cusp}(\Gamma_0(\mathfrak{a}),\mathcal{O})$. More precisely, when $\p$ is a prime ideal of residue degree one, our discussion above shows that 
 $$r = \textrm{rank}(\Gamma_0(\p)^{ab}), \ \ \ \text{for} \ \ -d=1,3$$
 and
 $$r+2 = \textrm{rank}(\Gamma_0(\p)^{ab}), \ \ \ \text{for} \ \ -d \not =1,3$$

I have computed the rank of $\Gamma_0(\p)^{ab}$ for prime ideals $\p$ of residue degree one and norm up to 45000 for $-d=1,3$, 30000 for  
$-d=2$, and 21000 for $-d=7,11$. I report on the distribution of prime levels according to the ranks in Table \ref{table: NRX} where I use $N_r(x)$ to denote the number of primes of residue degree one with norm $< x$ and such that $H^2_{cusp}(\Gamma_0(\mathfrak{a}),\mathcal{O})$ has rank $r$. It is curious that for all five $\mathcal{O}_d$, approximately $90 \%$ of the time the rank was 0. Note that in \cite{fgt}, Finis, Grunewald and Tirao used an efficient method which works with finite fields and approximated the ranks up to norm 60000 for $-d=1$. My computations for $-d=1$ agree with theirs in the part that they overlap except that the values $12113,12373$ are missing from their Table 10 and $12941$ is missing from their Table 11.

\begin{table}
\centering
\begin{tabular}{|c|c|c|c|c|c|c|c|c|c|c|} \hline
& \multicolumn{2}{|c|}{$d=-1$}&\multicolumn{2}{|c|}{$d=-2$}&\multicolumn{2}{|c|}{$d=-3$}
                                &\multicolumn{2}{|c|}{$d=-7$}&\multicolumn{2}{|c|}{$d=-11$} \\ \hline
$r$ & $N_r(45000)$& $\%$&$N_r(30000)$& $\%$&$N_r(45000)$& $\%$&$N_r(21000)$& $\%$&$N_r(21000)$& $\%$ \\ \hline
       0 &2061  &88.8 &  1480 & 91.8  & 2033  &87.4  & 1054   &89.5   & 1056&   89.7\\ \hline
       1 &177   &7.6  & 94   & 5.83   & 184   &7.91  & 89     &7.6    &  96 &   8.15\\ \hline
       2 &66    &2.8  & 31   & 1.92   & 82    &3.52  & 29     &2.5    &  22 &   1.90 \\ \hline
       3 &10    &0.4  & 5    & 0.31   & 21    &0.91  & 4      &0.3    &   2 &   0.17\\ \hline
       4 &4     &0.2  & 1    & 0.07   & 5     &0.22  & 0      &0      &   0 &    0\\ \hline
       5 &1     &0.04 & 0    & 0      & 0     &0     & 1      &0      &   0 &    0\\ \hline
       6 &1     &0.04 & 1    & 0.07   & 0     &0     & 0      &0.1    &   1 &    0.08\\ \hline
       7 &1     &0.04 & 0    & 0      & 1     &0.04  & 0      &0      &   0 &    0\\ \hline
$\geq 8$ &0     &0    & 0    & 0      & 0     &0     & 0      &0      &   0 &    0\\ \hline
$> 0$ &260   &11.2 & 132     & 8.2    & 293    &12.6  & 123    &10.5   & 121 &    10.3\\ \hline
\end{tabular}
\caption{distribution of  dimension of $\H^2_{cusp}(\Gamma_0(\mathfrak{p}),\C)$} 
\label{table: NRX}
\end{table}

It is believed that there are infinitely many prime ideals $\p$ of residue degree one such that $H^2_{cusp}(\Gamma_0(\mathfrak{p}),\C))=0$. 
On the other hand, in analogy with the conjecture that there are infinitely many elliptic curves over $\Q$ with prime conductor, see 
\cite{bs} p.97, it is reasonable to expect that there are infinitely many prime ideals $\p$ of residue degree one and 
$H^2_{cusp}(\Gamma_0(\mathfrak{p}),\C)) \not = 0$. 
In \cite{fgt}, in analogy with the distribution questions for elliptic curves (see \cite{bm}), the following question was posed (stated here in a slightly more general form).

\begin{question} Let $\mathcal{O}_d$ be the ring of integers of an imaginary quadratic number field. Is there constant $C_d$ such that the asymptotic relation
$$ \sum _{\p, \ \textbf{N}\p \leq x} \textrm{dim}\  H^1_{cusp}(\Gamma_0(\p),\C) \sim C_d \dfrac{x^{\frac{5}{6}}}{\textrm{log} \ x}$$
holds as $x$ goes to infinity, where the sum ranges over residue degree one prime ideals $\p \triangleleft \mathcal{O}_d$. 
\end{question}

Let us put $L_d(x):=\sum _{\p, \ \textbf{N}\p \leq x} \textrm{dim}\  H^1_{cusp}(\Gamma_0(\p),\C)$ for the ring $\mathcal{O}_d$ and 
$R(x):=x^{\frac{5}{6}} / \textrm{log} \ x$. Table \ref{table: elliptic} compares the two functions $L(x)$ and $R(x)$ within the range of my computations. \\

\begin{table}
\centering
\begin{tabular}{|c|c|c|c|c|c|} \hline
$x$ & $R(x)/ L_1(x)$ & $R(x) / L_2(x)$ & $R(x) / L_3(x)$ & $R(x) / L_7(x)$ & $R(x) / L_{11}(x)$ \\ \hline
3000 &  1.793 & 3.654 & 1.827 &   2.294 &    2.099 \\ \hline
6000 &  1.650 & 3.172 & 1.540 &   2.489 &    2.101 \\ \hline 
9000 &  1.720 & 3.334 & 1.435 &   2.462 &    2.281 \\ \hline 
12000  & 1.828 & 3.608 & 1.534 &  2.617 &    2.495 \\ \hline 
15000 &  1.927 & 3.610 & 1.524 &  2.662 &    2.731  \\ \hline 
18000 &  1.950 & 3.292 & 1.482 &  2.457 &    2.678  \\ \hline 
21000 &  1.912 & 3.114 & 1.575 &  2.464 &    2.642  \\ \hline 
24000 &  1.801 & 2.993 & 1.543 &  -     &   -  \\ \hline 
27000 &  1.782 & 3.000 & 1.594 &  -     &   -  \\ \hline 
30000 & 1.781 & 2.884  & 1.591 &  -     &   -  \\ \hline 
33000 & 1.830 & -     & 1.632 &  -     &   -  \\ \hline 
36000 & 1.831 & -     & 1.627 &  -     &   -  \\ \hline 
39000 & 1.825 & -     & 1.612 &  -     &   -  \\ \hline
42000 & 1.885 & -     & 1.628 &  -     &   -  \\ \hline
45000 & 1.887 & -     & 1.607 &  -     &   -  \\ \hline

\end{tabular}
\caption{data related to the asymptotics of nonvanishing of cuspidal cohomology}
\label{table: elliptic}
\end{table}

In \cite{cd}, Calegari and Dunfield constructed a family of commensurable arithmetic rational homology 3-spheres, that is, commensurable arithmetic Kleinian groups $\Gamma$ such that $H_1(\Gamma \backslash \mathbb{H}, \Q)\simeq H_1(\Gamma,\Q)=0$. In \cite{lmr}, Long, Maclahlan and Reid asked the question whether there are infinitely many commensurability classes of arithmetic rational homology 3-spheres. In the same paper, they posed the following two conjectures (they are slightly rephrased in an equivalent form that fits better with this paper). 

\vspace{.1 in}
\begin{conjecture} 
\begin{enumerate}
\item There exist infinitely many pairs of prime ideals $\{ \p_1, \p_2 \} \subset \Z[i]$ such that 
      $$H^2_{cusp}(\Gamma_0(\p_1 \p_2),\Q)=0.$$
\item Let $\p=(1+i)$. There are infinitely many prime ideals $\mathfrak{q} \subset \Z[i]$ with ${\bf N}\mathfrak{q} = 1 \mod 12$ such that 
      $$H^2_{cusp}(\Gamma_0(\p \mathfrak{q}),\Q)=0.$$      
\end{enumerate}
\end{conjecture} 

If the second conjecture holds, then using Jacquet-Langlands correspondence one gets (see \cite{lmr} p.29) a positive answer to the question of Long-MacLachlan-Reid stated above.

I computed the ranks of $\Gamma_0(\p \mathfrak{q})^{ab}$ where $\p=(1+i)$ and $\mathfrak{q} \subset \Z[i]$ prime with ${\bf N}\mathfrak{q} = 1 \mod 12$ of norm $\leq 14850$. There are 423 such prime ideals $\mathfrak{q}$, 245 of them satisfied the desired property that 
$$H^2_{cusp}(\Gamma_0(\p \mathfrak{q}),\Q)= \Gamma_0(\p \mathfrak{q})^{ab} \otimes \Q =0.$$ 
The uniform distribution of the primes with vanishing rank supports the second conjecture. 

\subsection{Asymptotics of torsion}

 Very recently M\"uller \cite{mueller} and Bergeron-Venkatesh \cite{bv} obtained significant results that relate the asymptotic behaviour of the size of the torsion in the homology of certain cocompact lattices in $\textrm{SL}_2(\C)$ to that of the volume of the associated 3-folds.
The following is a special case of the main result of Bergeron and Venkatesh.

\begin{theorem} Let $\{ \Gamma_n \}$ be a decreasing tower of cocompact arithmetic congruence subgroups of $\textrm{SL}_2(\C)$  
such that $\bigcap_n \Gamma_n = \{ 1 \}$. Put $X= \Gamma_1 \backslash \mathbb{H}$ where $\mathbb{H}$ is the hyperbolic 3-space. Then

$$ \lim_{n \rightarrow \infty} \dfrac{\log | H_1(\Gamma_n, E_{k,\ell})_{tor} |}{[\Gamma_1 : \Gamma_n]} 
                            = \dfrac{1}{6 \pi} \cdot c_{k,\ell} \cdot \textrm{vol}(X), \ \ \ \ k \not = \ell$$
where $E_{k,\ell}$ is standard module $\textrm{Sym}^k(\Z^2) \otimes \overline{\textrm{Sym}^{\ell}}(\Z^2).$ Here $c_{k,\ell}$ is a positive 
integer depending only on the non-equal parameters $k,\ell$.
\end{theorem}
Bianchi groups and their congruence subgroups are outside of the scope of these results as they are not cocompact. It is very interesting to numerically investigate whether similar asymptotic relations hold for them as well. 

I will compare the size of the torsion in the first homology and the volume of the associated 3-folds. As it is computationally costly to increase the weight, I will concentrate on increasing the level. 

Let $\Gamma_0(\p)$ where $\p$ is a prime ideal 
of $\mathcal{O}$ with residue degree one. Let $H_1(\Gamma_0(\p),\Z)_{tor}$ denote the torsion part of the first homology of $\Gamma_0(\p)$ with coefficients in $\Z$. Let $\text{vol}(\Gamma_0(\p) \backslash \H)$ denote the volume of the 3-fold 
$\Gamma_0(\p) \backslash \H$ where $\H$ is the hyperbolic 3-space. In light of the result of Bergeron-Venkatesh, the following is  
an interesting question.

\begin{question} With the notation of the above paragraph, is there a constant $C$, independent of $d$, such that the asymptotic relation
$$ \text{log}\abs{H_1(\Gamma_0(\p),\Z)_{tor}} \sim C \cdot \text{vol}(\Gamma_0(\p) \backslash \H)$$
holds as norm of the ideals $\p \triangleleft \mathcal{O}_d$, which are prime with residue degree one, tends to infinity?
\end{question} 

To computationally investigate the question, we need to approximate the volumes first.
Using the well-known formula, see \cite{gk},
$$\mathbb{V}_d:= \text{vol}(\psl_2(\mathcal{O}_d) \backslash \H)= \frac{\abs{\bigtriangleup_d}^{3/2}}{4 \pi^2} \ \zeta_{K_d}(2) \  ,$$
where $\bigtriangleup_d$ is the discriminant of the field $K_d$ and $\zeta_{K_d}$ is the Dedekind zeta function of $K_d$, we get
\begin{center}
\begin{tabular}{rl}
$\mathbb{V}_1 \simeq$ & $0.305321864725739671684867838311$ \\ 
$\mathbb{V}_2 \simeq$ & $1.00384100334119813727236488577$ \\  
$\mathbb{V}_3 \simeq$ & $0.169156934401608937503533759046$ \\
$\mathbb{V}_7 \simeq$ & $0.888914927816353263598904154202$ \\ 
$\mathbb{V}_{11} \simeq$ & $1.38260830790264587367165334450$ \\ 
\end{tabular}
\end{center}
Now for $\p \triangleleft \mathcal{O}_d$ prime of residue degree one over the rational prime $p$, we have
$$\text{vol}(\Gamma_0(\p) \backslash \H) = (p+1) \cdot \mathbb{V}_d.$$

I have collected data on the ratio of $\text{log}\abs{H_1(\Gamma_0(\p),\Z)_{tor}}$ to $\text{vol}(\Gamma_0(\p) \backslash \H)$ in the case of the five Euclidean $\mathcal{O}_d$. Ignoring the first 500 entries in each case, the average ratios read 
$0.054291, 0.053140, 0.055386, 0.053206, 0.053131$ respectively for $-d=1,2,3,7,11$. The range of my computations were 
up to norm $45000, 30000, 45000, 21000, 21000$ respectively. It is very significant that the ratio is very close to 
$$\frac{1}{6 \pi} \simeq 0.0530516476972984452562945877908 $$
which is the constant for the Lie group  $\textrm{SL}_2(\C)$ that appears in the above result of  Bergeron-Venkatesh. A small sample is given in Table \ref{table: log} where I use the convention $T_{\p}:=\text{log}\abs{H_1(\Gamma_0(\p),\Z)_{tor}}$ and 
$V_{\p}:=\text{vol}(\Gamma_0(\p) \backslash \H)$. The complete data is presented in my website.

\begin{small}
\begin{table}
\centering
\begin{tabular}{|c|c||c|c||c|c||c|c||c|c|} \hline
\multicolumn{2}{|c||}{$\mathcal{O}_1$} & \multicolumn{2}{|c||}{$\mathcal{O}_2$}  & \multicolumn{2}{|c||}{$\mathcal{O}_3$}  & \multicolumn{2}{|c||}{$\mathcal{O}_7$} & 
\multicolumn{2}{|c|}{$\mathcal{O}_{11}$} \\ \hline
$\textbf{N}\p $ & $T_{\p} / V_{\p}$ & $\textbf{N}\p $ & $T_{\p} / V_{\p}$ & $\textbf{N}\p $ & $T_{\p} / V_{\p}$ & $\textbf{N}\p $ & $T_{\p} / V_{\p}$ & $\textbf{N}\p $ & $T_{\p} / V_{\p}$  \\ \hline

44533	&	0.05342	&	27817	&	0.05338	&	44533	&	0.05288	&	20549	&	0.05337	&	19583	&	0.05368	 \\ \hline
44537	&	0.05391	&	27827	&	0.05247	&	44563	&	0.05288	&	20563	&	0.05414	&	19603	&	0.05324	 \\ \hline
44549	&	0.05250	&	27851	&	0.05300	&	44587	&	0.05559	&	20693	&	0.05411	&	19661	&	0.05212	 \\ \hline
44617	&	0.05467	&	27883	&	0.05463	&	44617	&	0.05279	&	20707	&	0.05410	&	19699	&	0.05244	 \\ \hline
44621	&	0.05509	&	27947	&	0.05282	&	44623	&	0.05352	&	20717	&	0.05297	&	19717	&	0.05331	 \\ \hline
44633	&	0.05390	&	27953	&	0.05221	&	44641	&	0.05558	&	20731	&	0.05269	&	19727	&	0.05327	 \\ \hline
44641	&	0.05317	&	27961	&	0.05439	&	44647	&	0.05581	&	20743	&	0.05353	&	19739	&	0.05346	 \\ \hline
44657	&	0.05203	&	28001	&	0.05342	&	44683	&	0.05509	&	20749	&	0.05172	&	19759	&	0.05385	 \\ \hline
44701	&	0.05520	&	28019	&	0.05161	&	44701	&	0.05791	&	20759	&	0.05121	&	19793	&	0.05410	 \\ \hline
44729	&	0.05351	&	28027	&	0.05359	&	44773	&	0.05280	&	20771	&	0.05204	&	19801	&	0.05296	 \\ \hline
44741	&	0.05355	&	28051	&	0.05231	&	44797	&	0.05357	&	20773	&	0.05411	&	19853	&	0.05157	 \\ \hline
44753	&	0.05533	&	28057	&	0.05238	&	44809	&	0.05239	&	20857	&	0.05258	&	19867	&	0.05442	 \\ \hline
44773	&	0.05604	&	28081	&	0.05214	&	44839	&	0.05606	&	20897	&	0.05290	&	19889	&	0.05311	 \\ \hline
44777	&	0.05573	&	28097	&	0.05198	&	44851	&	0.05300	&	20899	&	0.05470	&	19891	&	0.05352	 \\ \hline
44789	&	0.05172	&	28099	&	0.05353	&	44887	&	0.05332	&	20903	&	0.05326	&	19913	&	0.05324	 \\ \hline
44797	&	0.05480	&	28123	&	0.05271	&	44893	&	0.05427	&	20939	&	0.05348	&	19919	&	0.05191	 \\ \hline
44809	&	0.05220	&	28163	&	0.05233	&	44917	&	0.05308	&	20959	&	0.05395	&	19937	&	0.05389	 \\ \hline
44893	&	0.05476	&	28201	&	0.05140	&	44953	&	0.05433	&	20981	&	0.05243	&	19963	&	0.05383	 \\ \hline
44909	&	0.05227	&	28211	&	0.05325	&	44959	&	0.05292	&	20983	&	0.05425	&	19979	&	0.05266	 \\ \hline
44917	&	0.05281	&	28219	&	0.05185	&	44971	&	0.05547	&	21001	&	0.04985	&	19991	&	0.05346	 \\ \hline
44953	&	0.05441	&	28283	&	0.05312	&	44983	&	0.05481	&	21011	&	0.05473	&	20021	&	0.05318	 \\ \hline

\end{tabular}
\caption{the ratio the size of the torsion to the volume as level grows}
\label{table: log}
\end{table}
\end{small}

\begin{small}

\end{small}

 \vspace{.2 in}
  \noindent \textsc{Departament d'Algebra i Geometria, Facultat de Matem\'atiques, \\ 
           Universitat de Barcelona, Gran Via de les Corts Catalanes, 585, \\ Barcelona, Spain}
   
 \noindent \textit{E-mail address:} mehmet.sengun@uni-due.de\\

\end{document}